\RequirePackage{fix-cm}
\documentclass[11pt, a4paper]{article}
\usepackage{graphicx}
\usepackage{mathptmx}      
\usepackage{amsfonts}
\usepackage{amsmath}
\usepackage{amsthm}

\usepackage{subfig}
\usepackage{multirow}
\usepackage{longtable}
\usepackage{epstopdf}

\usepackage{algorithm,algpseudocode}

\usepackage{cite} 

\def\bd{\begin{description}}
\def\ed{\end{description}}
\def\beq{\begin{equation}}
\def\eeq{\end{equation}}
\def\bea{\begin{eqnarray}}
\def\eea{\end{eqnarray}}
\def\beas{\begin{eqnarray*}}
\def\eeas{\end{eqnarray*}}

\newtheorem{definition}{Definition}
\newtheorem{lemma}{Lemma}
\newtheorem{theorem}{Theorem}
\newtheorem{corollary}{Corollary}

\begin{document}


\title{\vspace{-2.6cm}Safe global optimization of expensive noisy black-box functions in the $\delta$-Lipschitz framework} 

\author{Yaroslav D. Sergeyev\thanks{Corresponding author: \texttt{yaro@dimes.unical.it}} \\[2pt]{\small University of Calabria,
Rende, Italy and Lobachevsky University, Nizhni Novgorod, Russia}
\and Antonio Candelieri \\[2pt]{\small University of Milano-Bicocca, Milan,
Italy} \and Dmitri~E.~Kvasov \\[2pt]{\small University of Calabria, Rende,
Italy and  Lobachevsky University, Nizhni Novgorod, Russia} \and
Riccardo Perego \\[2pt]{\small University of Milano-Bicocca, Milan, Italy}
}

\date{}

\maketitle

\begin{abstract}
In this paper, the problem of safe global maximization (it should
not be confused with robust optimization) of expensive noisy
black-box functions satisfying the Lipschitz condition is
considered. The notion ``safe" means that the objective function
$f(x)$ during optimization should not violate a ``safety" threshold,
for instance, a certain a priori given value $h$ in a maximization
problem. Thus, any new function evaluation (possibly corrupted by
noise) must be performed at ``safe points" only, namely, at points
$y$ for which it is known that the objective function $f(y) > h$.
The main difficulty here consists in the fact that the used
optimization algorithm should ensure that the safety constraint will
be satisfied at a point $y$ {\emph{before}} evaluation of $f(y)$
will be executed. Thus, it is required both to determine the safe
region $\Omega$ within the search domain~$D$ and to find the global
maximum within $\Omega$. An additional difficulty consists in the
fact that these problems should be solved in the presence of the
noise. This paper starts with a theoretical study of the problem and
it is shown that even though the objective function $f(x)$ satisfies
the Lipschitz condition, traditional Lipschitz minorants and
majorants cannot be used due to the presence of the noise. Then, a
$\delta$-Lipschitz framework and two algorithms using it are
proposed to solve the safe global maximization problem. The first
method determines the safe area within the search domain and the
second one executes the global maximization over the found safe
region. For both methods a number of theoretical results related to
their functioning and convergence is established. Finally, numerical
experiments confirming the reliability of the proposed procedures
are performed.

\vspace*{2mm} \textbf{Keywords:} {Safe global optimization;
expensive black-box functions; noise; Lipschitz condition; Machine
learning.}

\end{abstract}


\section{Introduction}
\label{intro} Many complex industrial applications   are
characterized by black-box, multiextremal, and expensive objective
functions that should be optimized (see
\cite{Ref5,Daponte:et:al.(1996),Ref1,Kvasov&Sergeyev(2015),PSKZ(2014),Sergeyev:et:al.(1999),book_Sergeyev&Kvasov(2017),Ref2,Ref3}).
The word \emph{expensive} means here that each evaluation of the
objective function $f(x)$ is a time consuming operation. Since
locally optimal solutions often do not give a sufficiently high
level of improvement w.r.t. a currently available solution obtained
by engineers using practical reasons, global optimization problems
are considered (see
\cite{Cavoretto:et:al.(2019),Ref5,Floudas&Pardalos(1996),Gergel2015,Horst&Pardalos(1995),Ref1,Pinter(1996),book_Sergeyev&Kvasov(2017),Strongin&Sergeyev(2000),Zilinskas2010IntervalAB,Vanderbei(1999),Zilinskas&Zhigljavsky(2016)}).
Unfortunately, a practical global optimization process is often
performed under a limited budget, i.e., the number of allowed
evaluations of $f(x)$ is fixed a priori and is not very high (see a
detailed discussion in \cite{Ref2}) requiring so an accurate
development of fast global optimization methods (see, e.\,g.,
\cite{Barkalov&Strongin(2018),Barkalov&Gergel(2016),Grishagin_Israfilov_Sergeyev(2018),Lera&Sergeyev2010b,PSKZ(2014),Paulavicius:et:al.(2020),book_Sergeyev&Kvasov(2017),Strongin&Sergeyev(2000),Zilinskas&Zhigljavsky(2016)}).

Recently, a class of  real applications having an additional
important constraint on the value of the objective function is under
an accurate study (see, e.\,g., \cite{Ref7,Ref6,Ref8} in the context
of Lipschitz optimization, control problems, reinforcement learning,
Bayesian optimization, etc.). This constraint requires that the
objective function $f(x)$ during optimization should not violate a
``safety'' threshold (that should not be confounded with robust
optimization, see, e.\,g., \cite{BenTal:et:al.(2009)}), for
instance, a certain a priori given value $h$ in a maximization
problem. Thus, any new function evaluation must be performed at
``safe points'' only, namely, at points $y$ for which it is known
that the objective function value will not violate the safety
threshold $h$, i.e., it should be $f(y)>h$. The main difficulty here
consists in the fact that the used optimization algorithm should
ensure that the safety constraint will be satisfied at a point $y$
\textbf{\emph{before}} evaluation of $f(y)$ will be executed.

This requirement is very important and represents a key difference
w.r.t. traditional constraint problems where constraints can be
checked and the algorithm then behaves in dependence on whether a
constraint was satisfied or not. In this kind of problems, it is
often allowed to evaluate the objective function $f(x)$ not only at
admissible but also at inadmissible points. In contrast, in the safe
optimization problems evaluation of $f(x)$ at inadmissible points is
strictly prohibited. Thus, while formulations using global
optimization with unknown constraints is well suited for
simulation-optimization problems, the safe global optimization
formulation is more appropriate for online control and online
learning and optimization problems (see, e.\,g.,
\cite{Ref7,Ref6,Ref8}).

For instance, let us consider, as an example of the safe
optimization problem, tuning parameters of a controller used for
automatic process control in the manufacturing industry, especially
in the pneumatic, electronic, robotic and automotive domains (see,
e.\,g., \cite{fiducioso2019safe,schillinger2017safe}). One of the
most widely used types of controllers is known as PID, according to
the three components: Proportional, Integral, and Derivative. A PID
controller implements a control loop mechanism based on feedback:
depending on the error value, computed continuously as the
difference between a desired setpoint and a measured process
variable, the PID controller applies a correction based on its
proportional, integral, and derivative terms.  Although a software
model of the system to control can be used to initially design the
PID controller, an expensive and prone-to-failures manual tuning
phase is needed to set up, on the real-life system, the optimal
values of its parameters from where the need to perform this phase
``safely'' arises. Depending on the characteristics of the system to
control, the safety of the controller is evaluated with respect to
one or more of the following issues: responsiveness of the
controller, overshooting of the desired setpoint by the measured
process variable as well as oscillatory behaviour around the
setpoint. In particular, overshooting and oscillatory behaviour
could lead to malfunctioning or breakage of the system to control.
Thus, searching for the optimal tuning of the PID parameters on a
real system, without considering ``safety'' of the system itself,
could damage it, in some cases irreparably.

In some sense, it is required both to determine the safe region
$\Omega$ within the search domain $D$ and to find the global optimum
within the safe region $\Omega$. Usually, at least one safe point is
known before the start of the optimization process (e.\,g., it is
taken using the current working configuration of the optimized
industrial system). Thus, it is required to invent a ``safe
expansion'' mechanism to extend the current  safe region from the
starting point in order to find the whole $\Omega$. It can happen
that, if the safe region $\Omega$ consists of several disjoint
subregions, those subregions which do not contain initially provided
safety points will be never found. In this case, it is not possible
to talk about the global optimum over the whole safe region $\Omega$
and a global optimum over the current safe subregion should be
found.

A further complication that is frequently present  in applied
optimization problems (see
\cite{Ref5,Floudas&Pardalos(1996),Ref1,Molinaro(2001),Pinter(1996),Strongin&Sergeyev(2000),Vanderbei(1999),Zilinskas&Zhigljavsky(2016)})
is the presence of noise affecting evaluations of the objective
function $f(x)$. These evaluations should remain safe even when they
are corrupted by noise. It is difficult to underestimate the
importance  of taking into consideration the presence of noise since
it can make invalid many assumptions (e.\,g., convexity,
derivability, Lipschitz continuity, etc.) usually done w.r.t.
optimized functions.  In spite of its crucial impact, noise is often
not considered in detail in safe global optimization problems
whereas  the afore mentioned approaches \cite{Ref7,Ref6,Ref8} do it
and provide some probability-based considerations on safety.
Precisely these papers have stimulated us to study safe global
optimization problems with noise.

Since Lipschitz continuity is a quite natural assumption for applied
problems (specifically for technical systems, see, e.\,g.,
\cite{Gillard&Kvasov(2016),Horst&Pardalos(1995),Kvasov&Sergeyev(2013),Pinter(1996),book_Sergeyev&Kvasov(2017),Strongin&Sergeyev(2000)}),
we consider here objective functions that satisfy the Lipschitz
condition over the search domain~$D$. Thus, our problem becomes
Lipschitz global optimization problem broadly studied in the
literature (see, e.\,g.,
\cite{Barkalov&Strongin(2018),Gergel:et:al.(2015L),Gillard&Kvasov(2016),Grishagin_Israfilov_Sergeyev(2018),Ref1,Kvasov&Sergeyev(2013),Pinter(1996),Sergeyev_Grishagin(2001),book_Sergeyev&Kvasov(2017),Strongin&Sergeyev(2000)}
and references given therein). In this paper,  we propose a new
Lipschitz-based safe global optimization algorithm, specifically
designed to work in the noisy setting. It is proved that, in spite
of the presence of the noise, our approach does not permit any
violation of the safety threshold. The only assumption made with
respect to the noise is its boundedness, with a maximal level of the
noise known a priori.

The remaining part of the manuscript is structured as follows.
Section~\ref{sec:2} contains   statement of the problem and its
analysis. Section~\ref{expan} presents a theoretical investigation
of a reliable expansion of the safe region and proposes an algorithm
realizing this expansion. Section~\ref{maxim} introduces a global
maximization algorithm working in the presence of the noise over the
found safe region. Section~\ref{sec:sec5} proposes three series of
numerical experiments confirming theoretical results and  showing a
reliable performance of the two introduced methods.
Section~\ref{sec:conclusions} concludes the paper.

\section{Statement of the problem and its analysis}
\label{sec:2} In order to start, let us present a general
formulation of the safe global optimization problem using the
Lipschitz framework in one dimension. As was already mentioned,
Lipschitz global optimization problems can be very often encountered
in applications even when $f(x)$ is univariate. Nowadays problems of
this kind with and without noise are under an intensive study (see,
e.\,g.,
\cite{Calvin&Zilinskas(2000),Calvin&Zilinskas(2005),Calvin:et:al.(2012),Casado_SIAM,Daponte:et:al.(1996),Kvasov&Mukhametzhanov(2018),Numta2019,Kvasov&Sergeyev(2012),Kvasov&Sergeyev(2015),Lera&Sergeyev(2013),Molinaro(2001),Pinter(1996),Sergeyev(1995),Sergeyev:et:al.(1999),Sergeyev:et:al(2001),Sergeyev:et:al.(2020),book_Sergeyev&Kvasov(2017),Sergeyev:et:al.(2017b),homogeneity,Sergeyev:et:al.(2016a)}).

To state the problem formally, let us suppose that a function $f(x)$
satisfies over a search domain $D=[a,b]$ the Lipschitz condition
 \beq
  |f(x_1) - f(x_2) | \le  L |x_1 - x_2 |, \hspace{5mm}
   0 < L < \infty,\hspace{5mm} x_1, x_2 \in   D,
\label{Lip}
 \eeq
with an a priori known Lipschitz constant
$L$\footnote{\label{note_L}Clearly, there can exist several
constants $\tilde{L}$ such that if they are placed in (\ref{Lip})
instead of $L$ the inequality will hold. Without loss of generality
we suppose hereinafter that $L>\min \tilde{L}$.}. Then, given a
safety threshold $h
>0$, it is required, in the presence of
noise ${\xi(x)}$,   to find an approximation of the point~$x^*$ and
an estimate of the corresponding value $g(x^*)$ such that
\begin{equation}
x^* = \underset{x \in \Omega \subseteq D}{\text{argmax}}
\hspace{3mm} g(x),  \hspace{1cm}g(x)=f(x)+\xi(x),
 \label{problem}
\end{equation}
where $\Omega$ is the safe region that can consist of several
disjoint subregions $\Omega_j, 1 \le j \le m,$ and the noise
${\xi(x)}$ is bounded by a known value $\delta$, i.e.,
\begin{equation}
|\xi(x)|\leq \delta,  \hspace{1cm} \delta > 0,
 \label{noise}
 \eeq
 \beq
 \Omega  = \{x: x \in D, g(x) \ge h \}, \hspace{5mm}
 \Omega=\cup_{j=1}^{m}\Omega_j, \hspace{1.5mm} \Omega_i\cap\Omega_j=\emptyset, \hspace{1mm} i\neq j.
 \label{omega}
\end{equation}
Thus, at each point $x$ the value $g(x)$ can belong to the set
 \beq
 G(x)= \{ y: y=f(x)+\xi(x), \xi(x) \in [-\delta,\delta] \}.
  \label{G(x)}
 \eeq

It should be stressed that in traditional noisy optimization
problems (see, e.\,g., \cite{Calvin&Zilinskas(2005)}) the maximizer
of $f(x)$, not of $g(x)$, is of interest. This is not the case in
safe optimization since the function $g(x)$ is measured and not
$f(x)$ and, therefore, it is important that precisely $g(x)$ would
be as larger than the threshold $h$ as possible.

\begin{figure}[t]
  \centering
  \includegraphics[width=1.0\linewidth]{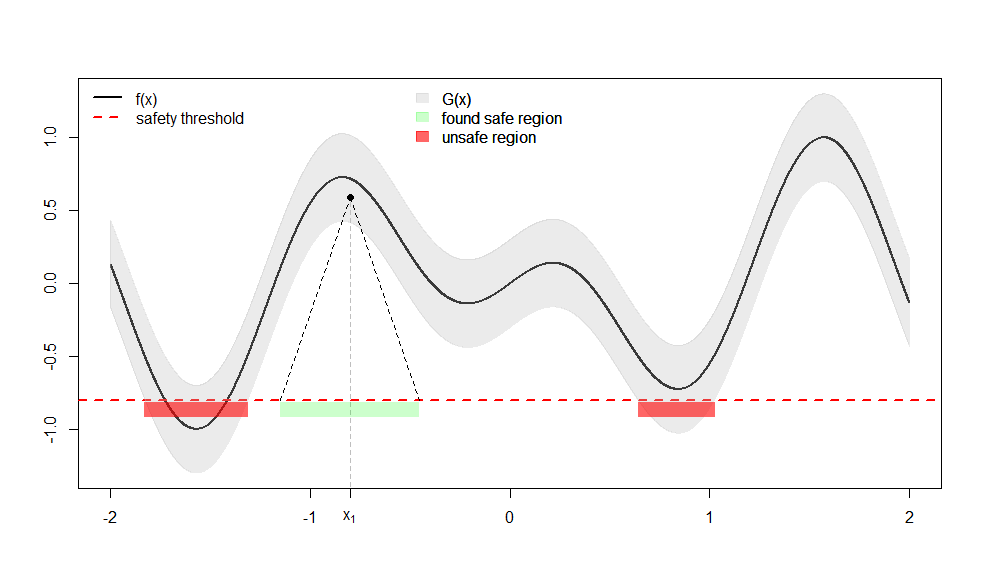}
\caption{An initial known safe point $x_1$, the initial safe region
(in green) found using  the function $\psi(x,x_1)$, and two unsafe
subregions (in red)}
  \label{fig:example_1}
\end{figure}

\begin{figure}[t]
  \centering
  \includegraphics[width=1.0\linewidth]{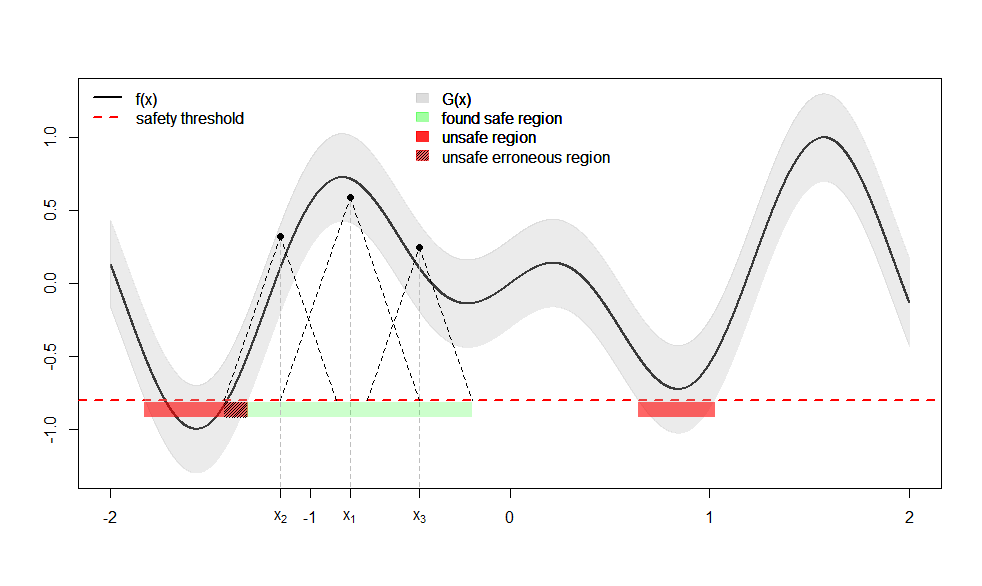}
\caption{An example of an error in determining the safe region using
functions $\psi(x,x_1), \psi(x,x_2),$ and $\psi(x,x_3)$ occurring
due to an inaccurate consideration of the noise}
  \label{fig:unsafe_expansion_example}
\end{figure}

In order to illustrate the problem (\ref{Lip})--(\ref{G(x)}), let us
consider Fig.~\ref{fig:example_1}. It shows the Lipschitz function
$f(x)=sin(x)\cdot cos(4x)$  having $L=4$ and defined over
$D=[-2,2]$. The safety threshold is $h=-0.8$ and the initial safe
point is $x_1=-0.8$. It can be seen that $f(x_1)\neq g(x_1) \in
G(x_1)$, in particular $g(x_1) < f(x_1)$.  The unsafe region is
shown in red; it consists of two subregions. Notice that in the
right-hand unsafe subregion $f(x)>h$ but due to the presence of the
noise this subregion becomes unsafe. The safe region, $Q$, consists
of three subregions $Q_1, Q_2,$ and $Q_3$. The presence of the safe
point $x_1$ in the central subregion $Q_2$ and the knowledge of the
Lipschitz constant $L$ together with supposition $g(x_1) < f(x_1)$
allow us to build the following Lipschitz minorant
 \[
 \psi(x,x_1) = g(x_1)-L|x_1-x|, \hspace{1cm} x \in D,
 \]
and to determine the initial safety region shown in green. It should
be expanded during optimization using the Lipschitz information,
however, since we have only one initial safe point at $Q_2$, only
this safe subregion   can be explored and the left-hand and
right-hand safe subdomains $Q_1$ and $Q_3$   will remain
undiscovered.

A  strategy for expanding the initial safe region   is illustrated
in Fig.~\ref{fig:unsafe_expansion_example}.  Clearly, the maximal
possible expansion of the currently found  safe region can be
obtained by evaluating the objective function at the points $x_2$
and $x_3$ being its extrema. Unfortunately, if the noise is not
taken into account appropriately, there is a risk to overestimate
the safe region. In real-life applications, at each safe point $y$
we do not know whether $g(y) \le f(y)$ or  $g(y) > f(y)$. For
instance, it can be seen in Fig.~\ref{fig:unsafe_expansion_example}
that in this example both $g(x_2)
> f(x_2)$ and $g(x_3) > f(x_3)$. Therefore,   functions
$\psi(x,x_2)$ and $\psi(x,x_3)$ constructed using the Lipschitz
constant $L$ of $f(x)$ at the points $x_2$ and $x_3$ are not
minorants for $g(x)$ anymore. In particular, this fact results in an
error in determining the safe region starting from the point $x_2$.
It can be seen that the region shown in
Fig.~\ref{fig:unsafe_expansion_example} in hatched red is unsafe but
it is considered by the described procedure of the expansion to be
safe. The source of this error is the fact that even though the
objective function $f(x)$ satisfies the Lipschitz condition
(\ref{Lip}), the function $g(x)$ is not Lipschitzian as it is shown
in Theorem~\ref{T1} below. To prove it, we need the following
definition (notice that similar but slightly different functions
have been considered in\,\cite{Vanderbei(1999)}).

\begin{definition}
A function $s(x)$ is called $\delta$-Lipschitz   over the interval
$D$ if it   satisfies the following condition
 \beq
  |s(x_1) - s(x_2) | \le  L |x_1 - x_2 | +\delta, \hspace{5mm}  0 < L < \infty,\,\, 0 < \delta    < \infty, \hspace{5mm} x_1, x_2 \in
  D.
\label{Lipdelta}
 \eeq
\end{definition}

\begin{theorem}\label{T1}
Suppose that the function $f(x)$ satisfies the Lipschitz condition
(\ref{Lip}) with a constant $L$, then the following two assertions
hold:

i. The function $g(x)$ from  (\ref{problem}) does not satisfy the
Lipschitz condition.

ii.  The function $g(x)$ from  (\ref{problem})  is
$2\delta$-Lipschitzian   where $\delta$ is from (\ref{noise}).
\end{theorem}

\begin{proof} Let us first prove the second assertion.  It follows
from definition of $g(x)$ in~(\ref{problem}), from  the Lipschitz
condition (\ref{Lip}),  and from the boundedness of noise
in~(\ref{noise}) that
 \[
  |g(x_1) - g(x_2) | = |f(x_1)+\xi(x_1)  - f(x_2)-\xi(x_2) | \le
  \]
 \beq
|f(x_1)   - f(x_2)  | + | \xi(x_1) | + |\xi(x_2) | \le L |x_1 - x_2
| +2\delta. \label{2delta}
   \eeq

Let us now consider the assertion (\emph{i}). If inequality
 \beq
  |g(x_1) - g(x_2) | \le  L |x_1 - x_2 |, \hspace{5mm}  0 < L < \infty,\hspace{5mm} x_1, x_2 \in
  D,
\label{Lipnot}
 \eeq
was satisfied (where $L$ is from (\ref{Lip})), then it would be that
$\lim_{x_1 \rightarrow x_2} |g(x_1) - g(x_2)| =0$. However, due to
(\ref{2delta}), $\lim_{x_1 \rightarrow x_2} |g(x_1) - g(x_2)|
=2\delta$ and, therefore, $g(x)$ does not satisfy (\ref{Lipnot}).
\end{proof}

\section{A reliable expansion of the  safe
region} \label{expan}

In order to provide a reliable expansion from the initial safe
region, it is necessary to ensure that at a new point $z$ chosen to
evaluate $g(x)$   condition (\ref{omega}) holds, i.e., $g(z) \ge h$.
Moreover, this condition should be  satisfied for any value of noise
from (\ref{noise}) and, as a consequence, for any value of $g(x) \in
G(x)$ from (\ref{G(x)}). Let us start to introduce such a mechanism
by supposing that $g(x)$ has been evaluated at several safe points
$x_i \in \Omega, 1 \le i \le k,$ and introduce the following
functions
 \beq
 \varphi(x,x_i) = g(x_i)-L|x_i-x|-2\delta, \hspace{1cm} x \in D,
 \label{fi}
   \eeq
where $\delta$ is from (\ref{noise}). Then, the following result
holds.

\begin{lemma}\label{L1}
Function $\varphi(x,x_i)$ satisfies Lipschitz condition (\ref{Lip}).
\end{lemma}

\begin{proof} To prove Lemma \ref{L1} it is required to show that
 \[
 | \varphi(x',x_i) - \varphi(x'',x_i)| \le L | x' - x'' |, \hspace{1cm} x', x'' \in
 D.
 \]
 For any arbitrary $x', x'' \in
 D$ it follows that
  \[
 | \varphi(x',x_i) - \varphi(x'',x_i)| = | g(x_i)-L|x_i-x'|-2\delta -g(x_i)+L|x_i-x''|+2\delta
 | =
 \]
   \[
   |\; ( L|x_i-x''|  -L|x_i-x'| ) \; | =  L\; |\;( |x_i-x'' -x'+x'|  -
   |x_i-x'|)\; | \le
 \]
   \[
   L \;| \;(\; |x'-x'' | +  |x_i-x' | - |x_i-x'|\; )\; |  =  L | x' - x'' |.
 \]
Thus, Lemma has been proven. \end{proof}

\begin{theorem}\label{T2}
Let us construct the following function
 \beq
 \Phi_k(x)= \max_{1 \le i \le k} \varphi(x,x_i), \hspace{1cm}  x \in D,
\label{Fi}
 \eeq
then for any value of noise $\xi(x) \in [-\delta,\delta]$ it follows
that $\Phi_k(x)$ is a minorant for   the function $g(x)$ satisfying
condition (\ref{2delta}) over the search region $D$, namely, it
follows
 \beq
\Phi_k(x)  \le  g(x)= f(x)+\xi(x), \hspace{1cm} \forall\,\, \xi(x)
\in [-\delta,\delta], \,\,x \in D. \label{minorant}
 \eeq
\end{theorem}

\begin{proof} Let us consider an arbitrary point $x \in D$. In
order to prove (\ref{minorant}), it is sufficient to show that
 \[
 \Phi_k(x)  \le \min_{\xi(x) \in [-\delta,\delta]} g(x) = f(x)-\delta.
 \]
 It
follows from (\ref{fi}) that there exists an index $j,\, 1 \le j \le
k,$ such that
 \beq
 \Phi_k(x)=\varphi(x,x_j) = g(x_j)- L |x_j-x|-2\delta.
 \label{safe1}
 \eeq
 Then, since $g(x)$ satisfies
condition (\ref{2delta}), we obtain
 \[
   g(x_j) -g(x) \le |   g(x_j) -g(x) |
  \le L|x_j-x|+2\delta.
 \]
 Notice that this estimate holds for any value of $g(x) \in G(x)$
 including, therefore, the minimal possible value of
 $g(x)=f(x)-\delta$. As a result we have
 \beq
 g(x_j) -(f(x)-\delta) \le      L|x_j-x|+2\delta.
 \label{safe2}
 \eeq
By adding and subtracting $f(x)-\delta$ to (\ref{safe1}) and
combining it with  (\ref{safe2}) we get
  \[
  \Phi_k(x) = g(x_j)
 -L|x_j-x|-2\delta + (f(x)-\delta)-(f(x)-\delta) \le
 \]
 \[   f(x)-\delta
 -L|x_j-x|-2\delta + L|x_j-x|+2\delta = f(x)-\delta.
 \]
Since the point $x$ has been chosen  arbitrarily, the last
inequality concludes the proof. \end{proof}

\begin{figure}[t]
  \centering
  \includegraphics[width=1.0\linewidth]{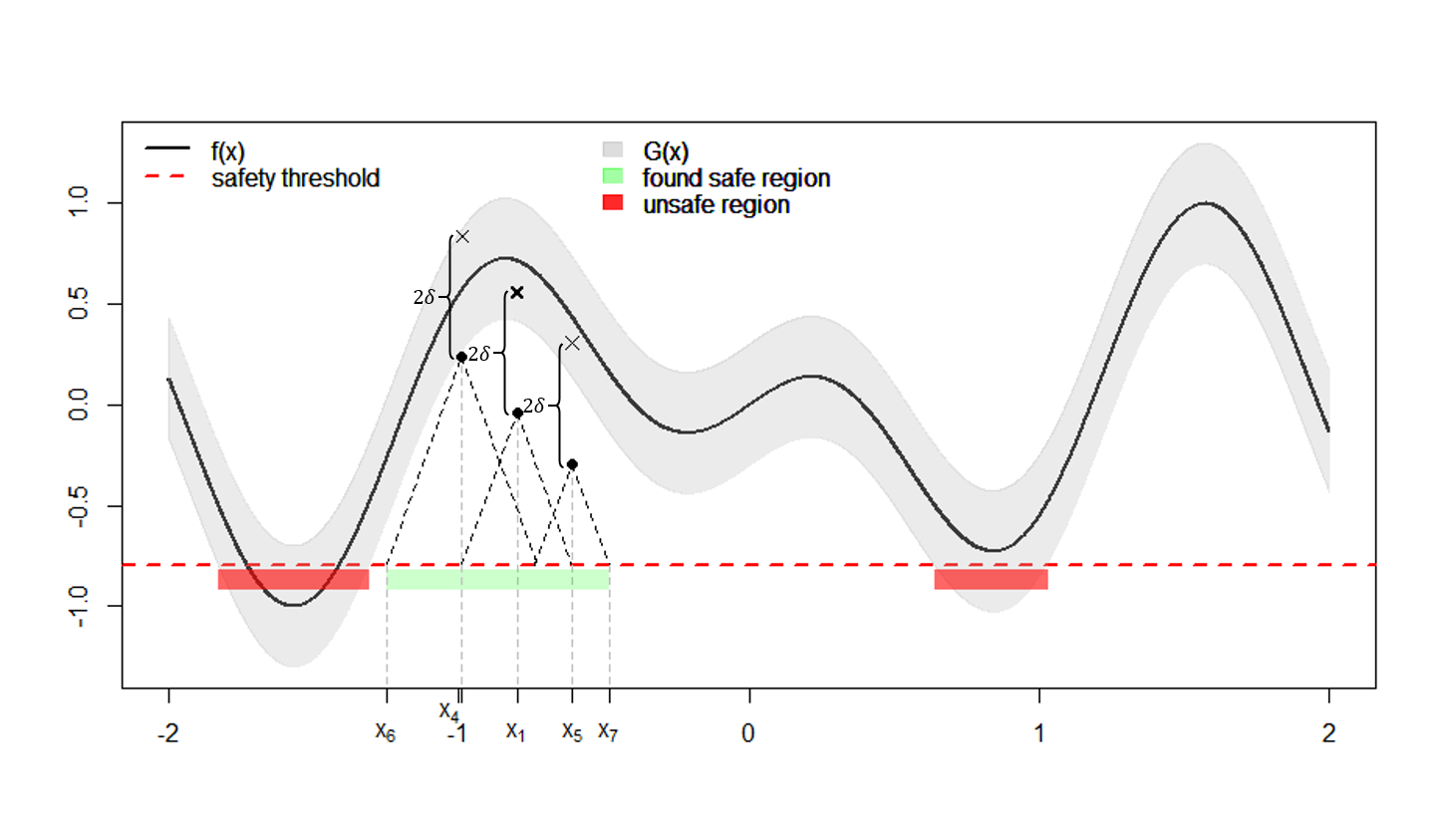}
\caption{\vspace{-1mm}The introduced in (\ref{fi})
$\delta$-Lipschitz minorants allow us to expand the initial safe
region  taking into account the noise correctly starting from the
initial safe point~$x_1$ and ensuring $\varphi(x,x_i) \le
f(x)-\delta, i=1,4,5$ (compare with
Fig.~\ref{fig:unsafe_expansion_example})}
  \label{fig:safe_expansion_example}
\end{figure}

\begin{corollary}\label{C1}
Function $\Phi_k(x)$ from  (\ref{Fi}) satisfies Lipschitz condition
(\ref{Lip}).
\end{corollary}

\begin{proof}  Due to Lemma \ref{L1}, functions   $\varphi(x,x_i),
1 \le i \le k,$ satisfy Lipschitz condition (\ref{Lip}). Since
$\Phi_k(x)$ is maximum of $k$ Lipschitz functions satisfying
(\ref{Lip}), this condition   holds for $\Phi_k(x)$, as well.
\end{proof}

Theorems~\ref{T1} and~\ref{T2} allow us to substitute unsafe
Lipschitz functions $\psi(x,x_i)$ with safe $\delta$-Lipschitz
minorants $\varphi(x,x_i)$ and give a direct suggestion on how the
initial safe region should be expanded. The process of expansion is
illustrated in Fig.~\ref{fig:safe_expansion_example}. We evaluate
$g(x)$ at the initial safe point $x_1$ (the result of evaluation
$g(x_1)$ is shown in Fig.~\ref{fig:safe_expansion_example}  by sign
``x'') and construct the minorant $\varphi(x,x_1)$ (the value
$\varphi(x_1,x_1)=g(x_1)-2\delta$ is shown by black dot). Then, by
using $\varphi(x,x_1)$, we obtain the initial  safe subregion $[x_4,
x_5]$. Clearly, since $\varphi(x,x_1) \le \psi(x,x_1)$, it follows
(cf. Fig.~\ref{fig:unsafe_expansion_example}) that $[x_4, x_5]
\subset [x_2, x_3]$. To expand the safe region we evaluate $g(x_4)$
and $g(x_5)$ and construct $\varphi(x,x_4)$ and $\varphi(x,x_5)$.
Theorem~\ref{T2} ensures that the obtained region $[x_6, x_7]$
(shown in green) is safe. The process is repeated and evaluations of
$g(x)$ are iteratively performed at the boundaries of the currently
found safe region as follows.

To continue our analysis let us suppose for simplicity that $\Omega$
is a simply connected region, i.e., $m=1$ in (\ref{omega}) (if this
is not the case then considerations analogous to what follows should
be executed for all zones $\Omega_j, 1 \le j \le m$). Assume  that
$g(x)$ has been evaluated at  safe points $x_i, 1 \le i \le k,$ and
let us indicate the points
 \beq
  \check{x}_k = \min \{ x_i: x_i \in \Omega, 1 \le i \le k \}, \hspace{5mm}
 \hat{x}_k = \max \{ x_i: x_i \in \Omega, 1 \le i \le k \},
  \label{safe3}
 \eeq
being the current minimal and maximal safe points, respectively.
Thus, the current found safe region is $[\check{x}_k,\hat{x}_k]$. By
using functions $\varphi(x,\check{x}_k)$ and $\varphi(x,\hat{x}_k)$
new safe points $\check{x}_{k+1}$ and $\hat{x}_{k+2}$ can be
obtained as solutions of equations
 \[
 g(\check{x}_k)-L(\check{x}_k-x)-2\delta = h, \hspace{1cm} x < \check{x}_k,
 \]
  \[
 g(\hat{x}_k)-L(x-\hat{x}_k)-2\delta = h, \hspace{1cm} x >
 \hat{x}_k,
 \]
 and calculated as follows
 \beq
 \check{x}_{k+1} = \check{x}_k - L^{-1}(g(\check{x}_k)- 2\delta-
 h),
 \hspace{5mm}
 \hat{x}_{k+2}  = \hat{x}_k +L^{-1}(g(\hat{x}_k)- 2\delta-
 h).
 \label{safe5}
 \eeq
It can be easily seen that the points  $\check{x}_{k+1}$ and
$\hat{x}_{k+2}$ are indeed safe. In fact, it follows from
Theorem~\ref{T2} that due to the construction
 \[
 h=\varphi(\check{x}_{k+1},\check{x}_k)  \le g(\check{x}_{k+1}), \hspace{1cm}
  h=\varphi(\hat{x}_{k+2},\hat{x}_k) \le g(\hat{x}_{k+2}).
 \]

Thus, the introduced  $\delta$-Lipschitz framework allows us to
construct  a safe expansion mechanism based on the knowledge of $L$
and $\delta$ and to obtain   estimates of the   left and right
margins of the safe region $\Omega$ (let us call them $l$ and $r$,
respectively).  However, there exists an additional trouble related
to situations taking place due to the presence of the noise when the
process of expansion approaches borders of the current safe region.
Let us illustrate this difficulty by returning once again to
Fig.~\ref{fig:safe_expansion_example}. It can be  seen that the
value $\varphi(x_4,x_4)$ is very close to the lower boundary of
$G(x_4)$ equal to $f(x_4)-\delta$ because $g(x_4)$ is very close to
the upper boundary of $G(x_4)$, namely, to $f(x_4)+\delta$. In
contrast, $\varphi(x_5,x_5)$ is very far away from   the lower
boundary of $G(x_5)$ equal to $f(x_5)-\delta$ because $g(x_5)$ is
very close to the lower boundary of $G(x_5)$, i.e., to
$f(x_5)-\delta$. Thus, the tightness of the minorant
$\varphi(x,x_i)$ depends on the level and sign of noise $\xi (x)$.
If $\xi (x) \rightarrow  \delta$, then the minorant is good, if $\xi
(x) \rightarrow  -\delta$, then the minorant is not accurate. This
fact can become crucial on the borders of the current safe region.
For instance, if the next function evaluation (see
Fig.~\ref{fig:safe_expansion_example}) is executed at the point
$x_6$ and $g(x_6)\approx f(x_6)+\delta$ then $\varphi(x_6,x_6)
\approx f(x_6)-\delta$, the constraint $\varphi(x_6,x_6) > h$ will
be satisfied an it will be possible to execute an additional (though
small) safe step on the left from $x_6$ and to enlarge the safe
region. In contrast, if $g(x_6)\approx f(x_6)-\delta$ then
$\varphi(x_6,x_6)=g(x_6)-2\delta$ will be less than $h$ and,
therefore, it will not be possible to expand the current safe
region.

This observation suggests that if at a current  point, $z$, that is
tested for expanding the safe region using the function
$\varphi(x,y), y \in \Omega,$ it follows that $\varphi(z,y) < h$,
then it is necessary to continue to evaluate $g(z)$ several times
trying to obtain a value close to $f(z)+\delta$ and, therefore, to
obtain the best possible minorant. This strategy is illustrated in
Fig.~\ref{fig:delta_safe_expansion} where the repeated evaluations
of $g(x)$ at the borders of the current safe region can be clearly
seen. These repetitions allow us to establish the borders with a
high precision.

\begin{figure}[t]
  \centering
  \includegraphics[width=1.0\linewidth]{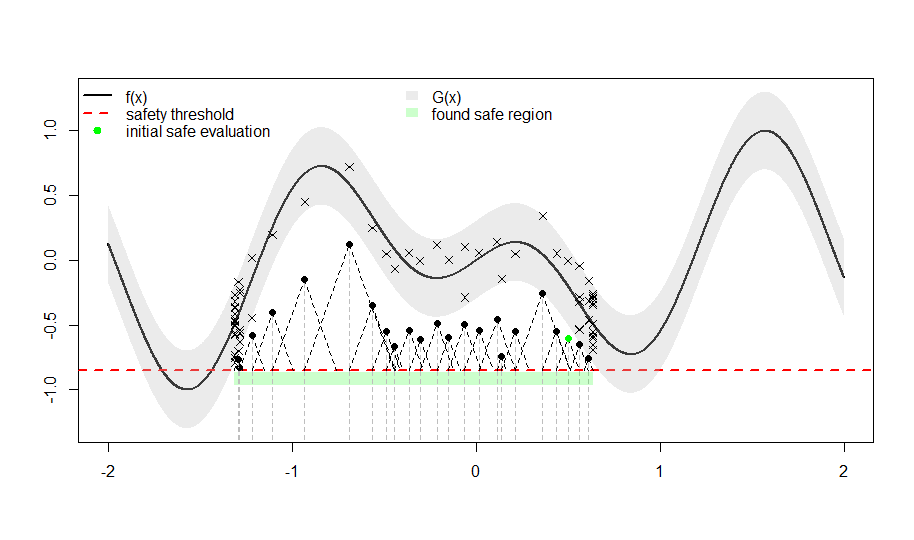}
\caption{\vspace{-1mm}The $\delta$-Lipschitz framework  reliably
expands the safe region. Repeated evaluations of $g(x)$ at the
borders of the current safe region can be clearly  seen.}
  \label{fig:delta_safe_expansion}
\end{figure}

In practice, the number of repetitions allowing us to obtain an
expansion can be very high and closer we are to the border of the
safe region smaller the expansion steps will become. Thus, it is
necessary to introduce a parameter $\nu$ fixing the maximal allowed
number of repetitions of evaluations of $g(x)$ at a point $z$ during
the search of the border and a  tolerance for the expansion process.
Recall that  we know that the level of noise is bounded by $\delta$.
Thus, if during  iterated evaluations of $g(z)$ two values, $y_1$
and $y_2$, $y_1 > y_2,$ have been obtained as results of these
evaluations then the following condition using an a priory given
accuracy $\sigma>0$ for stopping iterated evaluations of $g(x)$
 \beq
y_1 -y_2 \ge 2\delta-\sigma  \label{safe4}
 \eeq
can be used.

Let us illustrate the work of the  parameters $\nu$ and $\sigma$ by
the following examples. In the first of them shown in
Fig.~\ref{fig:delta_safe_expansion} the safe region for $h=-0.85$
has been identified using parameters $\nu=15$ and $\sigma=0.05\cdot
2\delta$. The search of the safe region has been stopped after 15
repetitions      at the lower and 15 repetitions at the upper margin
of the current safe region, i.e., the criterion on $\nu$ has stopped
the search.

\begin{figure}[t]
  \centering
  \includegraphics[width=1.0\linewidth]{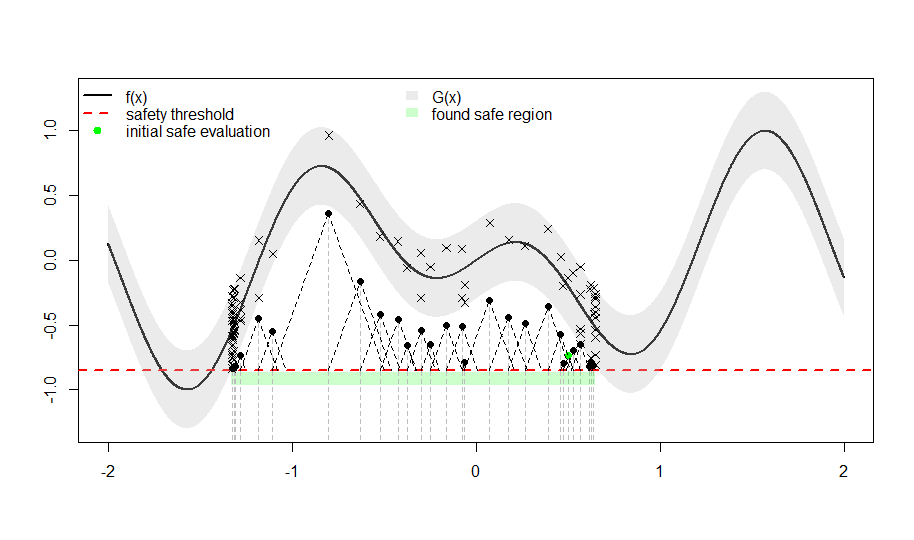}
\caption{\vspace{-1mm}The search of the safe region stops using the
criterion on $\sigma$ from (\ref{safe4}) with $\sigma=0.1\cdot
2\delta$. It can be seen by comparing with
Fig.~\ref{fig:delta_safe_expansion} that the density of evaluations
at the borders of the safe region is lower.}
  \label{fig:delta_safe_expansion1}
\end{figure}

In Fig.~\ref{fig:delta_safe_expansion1}, the same values of  $h$ and
$\nu$ have been used but $\sigma=0.1\cdot 2\delta$ has been taken.
Function evaluations at the lower and upper boundaries  of the safe
region were, respectively, 13 and 10. Namely, at both sides, the
search was stopped due to the criterion on $\sigma$. Notice that the
function $f(x)$ is the same in both experiments shown in
Figs.~\ref{fig:delta_safe_expansion}
and~\ref{fig:delta_safe_expansion1}. However, due to the presence of
the  noise the resulting functions $g(x)$ are different.

\begin{figure}[t]
  \centering
  \includegraphics[width=1.0\linewidth]{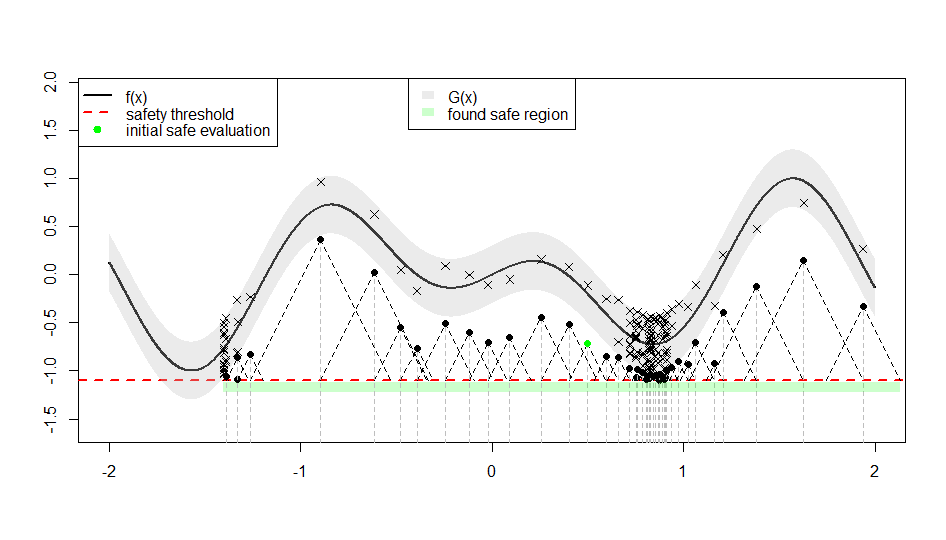}
  \caption{\vspace{-1mm}The initial safe point (shown  in green) is located quite far
from the upper margin of $\Omega$ and a local minimum of $f(x)$, in
between, is quite close  to the safety threshold $h$. At a
neighborhood of this local minimizer a huge number of function
evaluations is performed trying to observe a   value of $g(x)$ which
would guarantee  a reliable  expansion of the current safe region.
Since in this subregion $G(x)$ is very close to $h$, the size of
each new step of the expansion is quite small in this zone.}
  \label{fig:delta_safe_expansion_2}
\end{figure}

It should be stressed that, even if the objective function
evaluations are repeatedly performed at the boundaries of the
estimated safe region to better define them, there is no way to go
through unsafe zones to other safe subregions (that could contain
the desired global maximum) if these subregions do not contain
originally provided safe points. In fact, in this specific example
the safe subregion obtained starting from the initial safe point
$x=0.5$ shown in Fig.~\ref{fig:delta_safe_expansion1} in green does
not coincide with the entire safe region  $\Omega$. As a result, the
identified safe subregion does not contain the global maximum over
the whole   $\Omega$ since it is located within the unfound
subregion of  $\Omega$.

The importance of the strategy consisting in insisting on repeated
evaluations on the borders of the current search region is shown
even better in Fig.~\ref{fig:delta_safe_expansion_2} where the same
function $f(x)$ is considered but the safety threshold is decreased
to $h=-1.1$.  In this case, the threshold $h$ limits the safe region
on the left-hand side only, whereas  the right border of $\Omega$
coincides with the right margin of the search space~$D$.

It can be seen from Fig.~\ref{fig:delta_safe_expansion_2} that the
safe expansion mechanism is able to correctly estimate the part of
the safe region $\Omega$ containing the global optimum. This happens
in spite of the fact that the initial safe point (shown in
Fig.~\ref{fig:delta_safe_expansion_2} in green) is located quite far
from the upper margin of $\Omega$ and a local minimum of $f(x)$, in
between, is quite close  to the safety threshold $h$. At a
neighborhood of this local minimizer a huge number of function
evaluations is performed trying to observe a function value which
would guarantee  a reliable  expansion of the current safe region.
Since in these subregion $G(x)$ is very close to $h$, the size of
each new step of the expansion is quite small. In this example, the
safe subregion has been identified using $\nu=15$ and
$\sigma=0.1\cdot 2\delta$. The executed function evaluations at the
lower boundary  of the safe region were 6 (i.e., the process of the
expansion on this side was stopped due to the criterion on
$\sigma$). The expansion on the upper boundary  of $\Omega$ was
stopped because the right margin of the search region $D$ has been
reached.

A detailed description of Algorithm~1 executing the described above
procedure of the reliable expansion of the safe region can be found
in Appendix.

\section{Global maximization over the found safe region}
\label{maxim}

Let us suppose now that the $\delta$-Lipschitz expansion procedure
described in the previous section has stopped after evaluating
$g(x)$ at $k$ trial points and a desired approximation
$\widehat{\Omega}_k$ (indicated just $\widehat{\Omega}$ hereinafter)
of the safe region $\Omega$ has been obtained. After that, the
second phase of the algorithm, namely, global maximization over the
found safe region~$\widehat{\Omega}$, starts. For simplicity we
suppose here, as was done in the previous section, that
$\widehat{\Omega}$ is a simply connected region (all the subsequent
considerations in this section can easily be extended to the case of
disjoint subregions composing $\widehat{\Omega}$). Notice also that
the maximization problem~(\ref{problem}) over $\Omega$ is
substituted by its approximation, i.e., by the problem
\begin{equation}
x^* = \underset{x \in \widehat{\Omega} \subseteq \Omega \subseteq
D}{\text{argmax}} \hspace{3mm} g(x), \hspace{1cm}g(x)=f(x)+\xi(x),
 \label{problem2}
\end{equation}
The resulting problem (\ref{Lip}), (\ref{problem2}),
(\ref{noise})--(\ref{G(x)}) is solved by using all previously
executed evaluations  of $g(x)$ at the points $x_i \in
\widehat{\Omega}, 1 \le i \le k$ (see Algorithm~1 in Appendix). The
maximization procedure constructs a majorant $\Gamma_k(x), x \in
\widehat{\Omega},$ for $g(x)$ similarly to the construction of the
minorant $\Phi_k(x)$ from (\ref{Fi}). Recall that during the safe
expansion procedure there could be produced points $x_i$ where
$g(x)$ has been evaluated $\nu(x_i)$ times (where $\nu(x_i)$ in any
case is not superior to the maximal allowed number, $\nu$, of
repeated evaluations of $g(x)$ at a single point). This means that
different values $g_j(x_i)$ could be obtained during the $j$th
evaluation, $1 \le j \le \nu(x_i) \le \nu$ of $g(x)$ at a point
$x_i$,   $1 \le   i \le k$. Recall, that it was necessary to use
values
 \beq
 \hat{g}(x_i) = \max_{1 \le j \le \nu(x_i)} g_j(x_i), \hspace{5mm} 1 \le i \le
 k,
 \label{hat}
   \eeq
in order to build a better minorant $\Phi_k(x)$. Analogously, in
order to build a better majorant $\Gamma_k(x)$, we shall use values
 \beq
 \check{g}(x_i) = \min_{1 \le j \le \nu(x_i)} g_j(x_i), \hspace{5mm} 1
\le i \le k. \label{check}
   \eeq
Then, the function $\Gamma_k(x)$ is built as follows
 \vspace*{-1mm}
 \beq
  \Gamma_k(x)= \min_{1 \le i \le k} \gamma(x,x_i), \hspace{1cm}  x \in \widehat{\Omega},
\label{Gamma}
 \eeq
 \vspace*{-3mm}
 \beq
 \gamma(x,x_i) = \check{g} (x_i)+L|x_i-x|+2\delta, \hspace{1cm} x \in \widehat{\Omega},
 \label{gamma}
   \eeq
where $\delta$ is from (\ref{noise}). Both functions, $\Phi_k(x)$
and $\Gamma_k(x)$, are presented in Fig.~\ref{fig:Majorant} using
the data from Fig.~\ref{fig:delta_safe_expansion_2}.

\begin{figure}[t]
  \centering
  \includegraphics[width=1.0\linewidth]{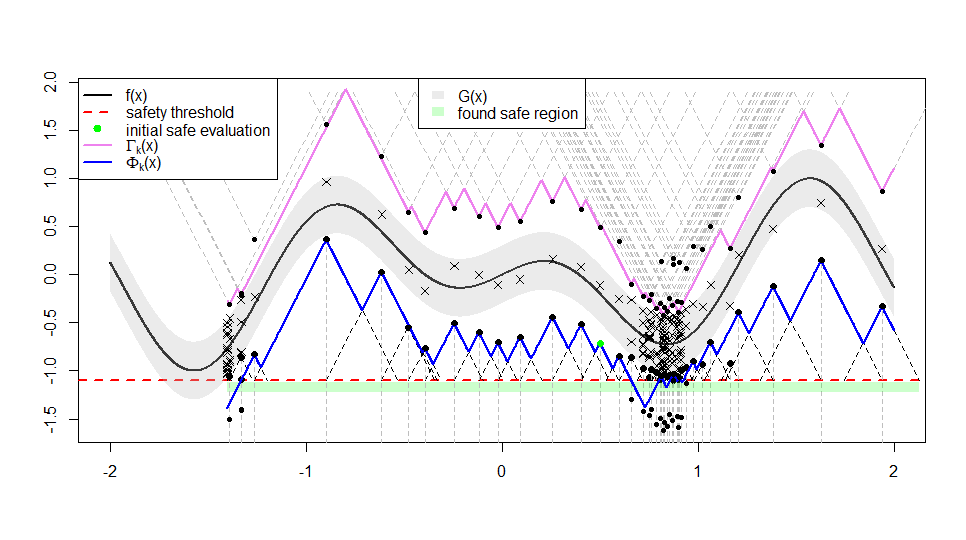}
\caption{The majorant $\Gamma_k(x)$ and the minorant $\Phi_k(x)$
constructed using the data from
 Fig.~\ref{fig:delta_safe_expansion_2}}
  \label{fig:Majorant}
\end{figure}

The following result ensures that
 \beq
\Gamma_k  \ge  g(x)= f(x)+\xi(x), \hspace{1cm} \forall\,\, \xi(x)
\in [-\delta,\delta], \,\,x \in \widehat{\Omega}.
 \label{majorant}
 \eeq

\begin{theorem}\label{T3}
For any value of noise $\xi(x) \in [-\delta,\delta]$ the function
$\Gamma_k(x)$ from (\ref{Gamma}) is a majorant for   the objective
function $g(x)$ satisfying (\ref{2delta}) over the search region
$D$, i.e., condition (\ref{majorant}) holds.
\end{theorem}

\begin{proof} The theorem follows immediately from Theorem~\ref{T2}. \end{proof}

\begin{corollary}\label{C2}
Function $\Gamma_k(x)$  from  (\ref{Gamma}) satisfies Lipschitz
condition (\ref{Lip}).
\end{corollary}

\begin{proof} It follows from considerations analogous to proof
of Lemma~\ref{L1} that functions $\gamma(x,x_i)$, $1 \le i \le k,$
satisfy Lipschitz condition (\ref{Lip}). Since $\Gamma_k(x)$ is
minimum of $k$ Lipschitz functions satisfying (\ref{Lip}), this
condition holds for $\Gamma_k(x)$, as well.
\end{proof}

In order to estimate the global maximizer of $g(x), x \in
\widehat{\Omega},$ we propose a procedure that generalizes  to the
case of maximization of noisy functions  the algorithm of Piyavskij
(see \cite{piya2}) proposed for Lipschitz optimization and working
without any noise (this algorithm is called PA hereinafter). The PA
works constructing a majorant $\Lambda_k(x)$ similar to
$\Gamma_k(x)$ at each iteration using previously executed
evaluations of the Lipschitz objective function $f(x)$. The absence
of the noise allows the PA to improve $\Lambda_k(x)$ after each
evaluation of $f(x)$ that is executed at a point
 \[
 x^{k+1}_{\Lambda_k}= \arg
\max_{x \in D} \Lambda_k(x).
 \]
 In other words, due to the Lipschitz
condition (\ref{Lip}) it follows that
 \beq
  f(x^{k+1}_{\Lambda_k}) \le   \max_{x \in D}  \Lambda_k(x), \hspace{5mm}
    \Lambda_{k+1}(x) \le \Lambda_k(x),\,\,\,  x \in D,
 \label{improve}
 \eeq
 and, as a consequence, after each new evaluation of $f(x)$ the
majorant $\Lambda_k(x)$ becomes closer and closer to the objective
function $f(x)$ allowing the PA to obtain a good approximation of
the global maximum.

In contrast, in the noisy framework we deal with, a condition for
$\Gamma_k(x)$  similar to (\ref{improve}) cannot be written down
since the noise (see Fig.~\ref{fig:delta_safe_expansion_2}) can
provoke a situation where
  \beq
  g(x^{k+1}) >   \max_{x \in \widehat{\Omega}}  \Gamma_k(x), \hspace{5mm} x^{k+1}=
   \arg \max_{x \in \widehat{\Omega}}
  \Gamma_k(x).
   \label{noimprove}
 \eeq
In cases where the condition  (\ref{noimprove}) holds, the
evaluation of $g(x^{k+1})$ does not improve $\Gamma_k(x)$ and leads
to the equivalence $\Gamma_{k+1}(x)=\Gamma_k(x)$. Let us recall that
we have already faced similar situations before, when we got results
of evaluation of $g(x)$ that did not allow us to enlarge the current
safe region. Therefore, the remedy will be the same, i.e., execution
of $\nu(x^{k+1}) \le \nu$ repeated evaluations of $g(x)$ at the
point $x^{k+1}$ from (\ref{noimprove}) until either the number of
possible repetitions $\nu$ will be reached or a value  less than
$\max_{x \in \widehat{\Omega}}  \Gamma_k(x)$ will be produced. In
the former case the search stops and in the latter one   (similarly
to (\ref{check})) the value
 \beq
 \check{g}(x^{k+1}) = \min_{1 \le j \le \nu(x^{k+1})} g_j(x^{k+1}) <
  \max_{x \in \widehat{\Omega}}  \Gamma_k(x) \label{check_2}
   \eeq
is accepted and $\Gamma_{k+1}(x) \le \Gamma_k(x)$ is built.

In practice, the operations of building $\Gamma_k(x)$ from the trial
data $(x_i, \check{g}(x_i)), 1 \le i \le k,$ provided by Algorithm~1
are executed as follows. To each trial point $x_i$ we associate a
value
 \beq
 z_i = \min_{1 \le j \le k} \gamma (x_i,x_j), \hspace{2mm}   1 \le i \le k.
 \label{z_i}
   \eeq
   Then, it follows from the Lipschitz condition (\ref{Lip}) that
 \[
 \max_{x \in [x_{i-1},\;x_i]} \Gamma_k(x)  =  \max_{x \in
 [x_{i-1},\;x_i]}\min
 \{ z_{i-1}+L(x-x_{i-1}), z_i+L(x_i-x) \} =
 \]
 \beq
= R_i =
 0.5(z_{i-1}+z_i) + 0.5L(x_i-x_{i-1}), \,\,\,  2 \le i \le k, \label{R_i}
   \eeq
   and it is reached at the point
   \beq
 \bar{x}_i = 0.5 (x_i+x_{i-1}) +  0.5(z_i-z_{i-1})/L, \label{bar_x}
   \eeq
where hereinafter the value $R_i$ is called \emph{characteristic} of
the interval $[x_{i-1},x_i]$. As a result, it follows that
 \beq
  \max_{x \in \widehat{\Omega}}   \Gamma_k(x) = \max_{2 \le i \le k}
  R_i
  \label{maxR}
   \eeq
and, therefore, the point $x^{k+1}$ from (\ref{noimprove}) can be
calculated as
 \beq
 x^{k+1} = \bar{x}_t, \hspace{5mm} t= \arg \max_{2 \le i \le k}
  R_i. \label{newpoint}
   \eeq
If the interval $[x_{t-1},x_t]$ is larger than the preset accuracy
$\varepsilon$ (i.e., the stopping rule is not satisfied), the
process of global maximization tries to update the majorant
$\Gamma_{k}(x) $ by obtaining a value $\check{g}(x^{k+1})$
satisfying (\ref{check_2}). In case this is not possible after~$\nu$
evaluations of $g(x^{k+1})$, the algorithm  stops. Otherwise, the
interval $[x_{t-1},x_t]$ is subdivided in two subintervals
$[x_{t-1},x^{k+1}]$ and $[x^{k+1},x_t]$,   the number of points
where $g(x)$ has been evaluated becomes $k+1$, the intervals having
numbers larger than $t$ are renumbered and the two new intervals
become $[x_{t-1},x_t]$ and $[x_t,x_{t+1}]$, respectively. Then, the
value $\check{g}(x^{k+1})$ is assigned to $z_t$   and  all the
values $z_i$ are renewed as follows
 \beq
  z_i = \min  \{ z_i, \gamma (x_i,x_t) \},\hspace{5mm} 1 \le i \le k+1.
 \label{z_i2}
   \eeq
This concludes construction of the new majorant  $\Gamma_{k+1}(x) $
and the optimization process can be repeated. A detailed description
of Algorithm~2 executing the global maximization over
$\widehat{\Omega}$ is given in Appendix (the maximization is
performed separately at each subregion of $\widehat{\Omega}$, but it
can also be executed simultaneously over all disjoint subregions of
the safe region, thus accelerating the search).

After satisfaction of one of the two stopping rules of Algorithm~2,
a lot of information regarding the problem (\ref{Lip}),
(\ref{problem}), (\ref{noise})--(\ref{G(x)}) is returned. First of
all, the found safe region $\widehat{\Omega}$ is provided. Then, the
largest found value  $g^{*}_k$ and the corresponding point
$x^{*}_k$, namely,
 \beq
 g^{*}_k = \max  \{ \hat{g}(x_i):\,\, 1 \le i \le k \},
 \hspace{5mm} x^{*}_k = \arg \max  \{ \hat{g}(x_i):\,\, 1 \le i \le k
 \}
 \label{g*k}
   \eeq
can be taken as an estimate of the global maximum for the safe
optimization problem (\ref{Lip}), (\ref{problem2}),
(\ref{noise})--(\ref{G(x)}). Then, the minorant $\Phi_k(x)$ and the
majorant $\Gamma_{k}(x) $ for functions $f(x)$ and $g(x)$ are
supplied, i.e.,
 \[
\Phi_k(x) \le f(x) \le \Gamma_{k}(x),  \hspace{5mm} \Phi_k(x) \le
g(x) \le \Gamma_{k}(x), \hspace{5mm} x \in \widehat{\Omega},
 \]
where   both $\Phi_k(x)$ and   $\Gamma_{k}(x) $ are Lipschitz
piece-wise linear functions (see Fig.~\ref{fig:Majorant}). Their
simple structure allows one to delimit easily areas of
$\widehat{\Omega}$ where the global maximizers of $g(x)$ and $f(x)$
cannot be located. In fact, the set
 \[
 N^k_g = \{ x : \Gamma_{k}(x) < g^{*}_k   \}
 \]
cannot contain the global maximizer of $g(x)$ and
 \[
 N^k_f = \{ x : \Gamma_{k}(x) < g^{*}_k - \delta \}
 \]
cannot contain the global maximizer of $f(x)$  (see
Fig.~\ref{fig:fig8} for illustration). The following lemma
establishes that each new trial point produced by Algorithm~2 is
chosen with the tentative to reduce the area    $\widehat{\Omega}
\setminus N^k_g $ containing the global maximizer of $g(x)$. This
reduction becomes effective in case the condition (\ref{noimprove})
does not hold at the chosen point $x^{k+1}$.

\begin{figure}[t]
    \centering
    {
    \includegraphics[width=1.0\textwidth, height=0.65\textwidth]{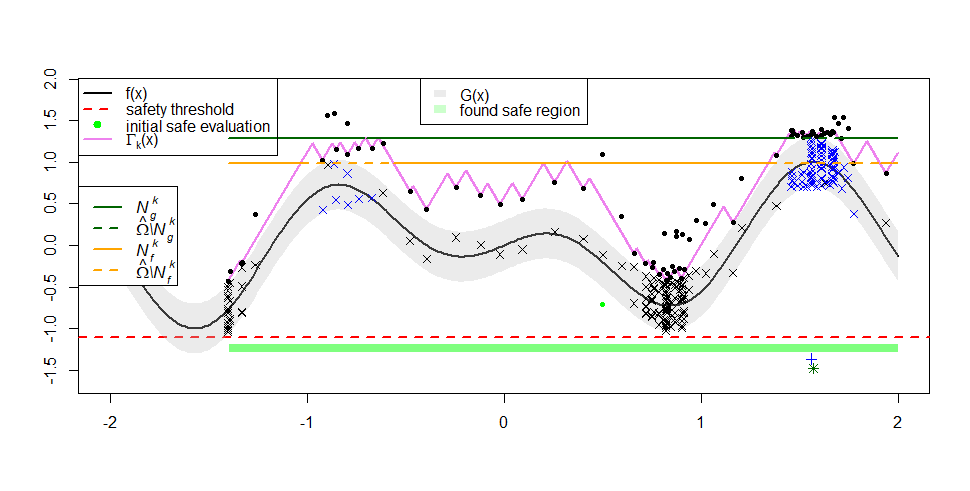}}
\caption{Results of the global maximization performed by Algorithm~2
using the data from Fig.~\ref{fig:Majorant} within  the found  safe
region $\widehat{\Omega}$. Algorithm~2 returns also regions $N^k_g$
and $N^k_f$ that cannot contain the global maximizers of $g(x)$ and
$f(x)$, respectively, and regions $\widehat{\Omega} \setminus N^k_g
 \subset \widehat{\Omega} \setminus N^k_f$ that contain the global solutions.
  }
    \label{fig:fig8}
\end{figure}

\begin{lemma}\label{L2}
All   trial points produced by Algorithm 2 at any iteration number
$k > 1$ will belong to the area   $\widehat{\Omega} \setminus
N^k_g$.
\end{lemma}

\begin{proof} Due to its definition, the area $N^k_g$ contains
points $x$ for which $\Gamma_{k}(x) < g^{*}_k$. It follows from
(\ref{maxR}) and (\ref{newpoint}) that the new trial point $x^{k+1}$
will be chosen by Algorithm~2 at the point corresponding to $\max_{x
\in \widehat{\Omega}} \Gamma_k(x)$. The fact that this value is
strictly larger than $g^{*}_k$ concludes the proof. \end{proof}

The introduced  Algorithm~2 works with the function $g(x)$ and looks
for its global maximizer $x^*$ from (\ref{problem2}). However,
sometimes the interest of the person executing the optimization is
related to the original function $f(x)$. Clearly, due to the noise,
the global maxima of $g(x)$ and   $f(x)$ can be different both in
value and in location. The following theorem establishes conditions
where Algorithm~2 can localize   not only the area $\widehat{\Omega}
\setminus N^k_g$ containing  the global maximum of the function
$g(x)$ but also can indicate a neighborhood of the global maximizer
$X^*$ of $f(x)$ over\,$\widehat{\Omega}$.

\begin{theorem}\label{T4}
Suppose that:

i. The function $f(x)$ satisfies the Lipschitz condition
(\ref{Lip}).

ii. The function $f(x)$ has the only global maximizer $X^*$ over
$\widehat{\Omega}$   and for any local maximizer $x'$ of $f(x)$ such
that $x'\neq X^*,\, x' \in \widehat{\Omega},$ it follows that
 \beq
 f(X^*) > f(x')+2\delta +  \Delta,
  \label{Delta}
   \eeq
where $\Delta>0$ is a fixed finite number and $\delta$, as usual, is
the maximal level of the noise  from~(\ref{noise}).

iii. Parameters of Algorithm~2 are chosen as follows:
$\varepsilon=0$, $\nu=\infty$.

Then during the work of  Algorithm~2 there exists an iteration
number $d^*$ such that for $k>d^*$ the only local maximizer of
$f(x)$ located at the region $\widehat{\Omega} \setminus N^k_g$
is~$X^*$.
\end{theorem}

\begin{proof} Supposition (iii) regarding parameters of
Algorithm~2  means that the majorant $\Gamma_{k}(x)$ will be
improved at each iteration at each new trial point $x^{k+1}$ because
the number of allowed repeated evaluations of $g(x^{k+1})$ is
$\nu=\infty$ and, therefore, the algorithm will not stop due to this
stopping rule (see lines 48--49 of Algorithm 2). Then, since we have
$\varepsilon=0$ in the second stopping rule, the algorithm will not
stop due to this rule (see lines 43 and 70 of Algorithm~2), as well.
Thus, supposition (iii) ensures that the sequence of trial points
generated by Algorithm~2 is infinite and the majorant
$\Gamma_{k}(x)$ is improved at each iteration.

Since $f(x)$ satisfies the Lipschitz condition (\ref{Lip}) and due
to the finiteness of $\Delta$ it follows from supposition (ii) and
condition (\ref{Delta}) that there exists a   neighborhood
$\omega(X^*)$ of the point~$X^*$ such that
 \[
 \omega(X^*)=(X^*-\omega,X^*+\omega), \hspace{5mm} \omega
 \ge  (f(X^*)-\Delta)/L,
 \]
 \beq
  f(x) > f(x')+2\delta, \hspace{5mm} x \in \omega(X^*),
 \label{safeT4.1}
   \eeq
and there are no other local maximizers of $f(x)$ belonging to
$\omega(X^*)$.

By taking into account the definition of $g(x)$ in (\ref{problem})
and (\ref{noise}), condition (\ref{safeT4.1}) can be re-written in
the form
 \[
   g(x) \ge f(x) - \delta > f(x')+ \delta \ge g(x'), \hspace{5mm} x \in
   \omega(X^*),\,\, x' \in \widehat{\Omega},
   \]
meaning that for any point $x \in \omega(X^*)$ and any local
maximizer $x'\neq X^*,\, x' \in \widehat{\Omega},$ it follows that
the value of $f(x)$ decreased by the maximal level of the noise is
larger than the value of $f(x')$ increased by the highest level of
the noise and, as a result, we have
 \beq
   g(x)   >   g(x'), \hspace{5mm} x \in \omega(X^*),\,\,
   x' \in \widehat{\Omega}, \,\, x'\neq X^*.
   \label{safeT4.2}
   \eeq

On the other hand, Algorithm 2 works in such a way that the majorant
$\Gamma_{k}(x)$ is improved at each iteration and the point of its
improvement corresponds to its maximum (see (\ref{maxR}) and
(\ref{newpoint})). This means  that $\max_{x  \in \widehat{\Omega}}
\Gamma_{k}(x)$ decreases with the increase of the iteration
number~$k$ (notice that this decrease is not strict since there can
be situations where $\max_{x  \in \widehat{\Omega}} \Gamma_{k}(x)$
is reached at several points). The decrease of  $\max_{x  \in
\widehat{\Omega}} \Gamma_{k}(x)$ and the presence of a finite
neighborhood $\omega(X^*)$ where all points $x \in \omega(X^*)$
satisfy (\ref{safeT4.2})  mean that there will exist an iteration
number $\tilde{d}$ such that the new point $x^{\tilde{d}}$  will
belong to $\omega(X^*)$ and, therefore, the value of the function
$g(x^{\tilde{d}})$ will satisfy
 \beq
 g(x^{\tilde{d}}) > \max_{x' \in \widehat{\Omega}} g(x')
  \label{safeT4.3}
   \eeq
for all local maximizers $x'\neq X^*,\,x' \in \widehat{\Omega}$. Due
to Lemma~\ref{L2}, new trials  with the numbers $k> \tilde{d}$ will
be executed at regions where $\Gamma_{k}(x) > g(x^{\tilde{d}})$.
Therefore, since (\ref{safeT4.3}) holds and due to the already
mentioned decrease of  $\max_{x  \in \widehat{\Omega}}
\Gamma_{k}(x)$, after a finite number  $\hat{d}$  of iterations of
Algorithm~2 new trials will not be placed outside $\omega(X^*)$. In
fact, suppose that there exists a local maximizer $x' \in
\widehat{\Omega}$ that does not belong to  $\omega(X^*)$, at the
$k$-th iteration   $x' \in [x_{i(k)-1},x_{i(k)}]$, and new trials
are placed infinitely many times within this interval. It follows
from (\ref{bar_x}) and footnote~\ref{note_L} (see
section~\ref{sec:2}) that in this case
 \[
  \lim_{k\rightarrow\infty} x_{i(k)} - x_{i(k)-1} = 0
 \]
and, therefore,
  \[
  \lim_{k\rightarrow\infty} x_{i(k)-1}  = x', \hspace{5mm}
  \lim_{k\rightarrow\infty} x_{i(k)}  = x'.
 \]
As a result, for the characteristic $ R_{i(k)}$ from (\ref{R_i}) of
this interval we have
 \[
  \lim_{k\rightarrow\infty} R_{i(k)}   \le g(x')
 \]
where the sign $\le$ here is due to the rule (\ref{z_i2}). However,
since $g(x^{\tilde{d}}) > g(x')$, the new trial cannot be placed
within the interval  $[x_{i(k)-1},x_{i(k)}]$ and, therefore, our
supposition that new trial points will be placed in this interval
infinitely many times is false.

Thus, by taking $d^*=\tilde{d}+\hat{d}$ and reminding that
$\omega(X^*)$ does not contain other local maximizers we conclude
the proof. \end{proof}

\section{Numerical experiments} \label{sec:sec5}

The experiments have been organized on three test cases,  where the
parameters of the $\delta$-Lipschitz framework have been set as
follows: both $\delta$ and $h$ have been set at $10\%$ of the
min-max range of $f(x)$, $\varepsilon=0.001$, $\nu=15$, and
$\sigma=0.10$. For all test cases,    initial safe points have been
chosen randomly.

\begin{itemize}
\item \textbf{Test case 1.}
This test case consists of four test problems. The safe expansion
phase can lead to the identification of a unique safe region. Table
\ref{tab:TC1} presents the structure of the test problems used in
this test case. Fig. \ref{fig:TestCase1Part1} shows for each test
problem the final state of each phase in optimization processes
(safe expansion -- phase 1 and global maximization -- phase 2).

\item
\textbf{Test case 2.} This test case consists of six test problems.
The safe expansion phase can lead to the identification of the
maximum of two disconnected safe regions. Table \ref{tab:TC2}
presents the structure of the test problems used in this test case.
Figs. \ref{fig:TestCase2Part1}--\ref{fig:TestCase2Part2} show for
each test problem  the final state of each phase in optimization
processes (safe expansion -- phase 1 and global maximization --
phase 2).

\item
\textbf{Test case 3.} This test    case consists of eight test
problems. The safe expansion phase can lead to the identification of
multiple disconnected safe regions. Table \ref{tab:TC3} presents the
structure of the test problems used in this test case. Figs.
\ref{fig:TestCase3Part1}--\ref{fig:TestCase3Part2} show for each
test problem  the final state of each phase in optimization
processes (safe expansion -- phase 1 and global maximization --
phase 2).
\end{itemize}

Table \ref{tab:Results} presents results of the numerical
experiments and
Figs.~\ref{fig:TestCase1Part1}--\ref{fig:TestCase3Part2} illustrate
them. Figures related to Phase~1 show initial safe points by green
dots and results of evaluations by symbol `x'; they present also the
constructed minorant $\Phi_k(x)$ and the found safe region. Figures
related to Phase~2 show results of evaluations executed during the
global maximization by blue symbols `x'. The best found point is
shown by symbol `+' and the global maximizer by symbol `*'. These
figures show also the constructed majorant $\Gamma_k(x)$ together
with the functions $\gamma(x,x_i)$ from (\ref{gamma}) used to build
$\Gamma_k(x)$. The black dots above the majorant represent values
that either have not improved the majorant or  better values (see
(\ref{z_i2})) than these ones have been obtained during the
maximization process.

\begin{table}[!t]
{\footnotesize \caption{Test Case 1} \label{tab:TC1}
\begin{center}
{\renewcommand{\arraystretch}{2.4}
\centering
\vspace*{-2mm} \begin{tabular}{|c|p{3.2cm}|p{1.20cm}|c|c|p{0.60cm}|} \hline 
\bf Problem & \bf \hspace*{5mm}Lipschitz Function & \bf Interval &
\bf \emph{L} & \bf Threshold & \bf Ref. \\ \hline 1 & \(f(x) =
-\frac{1}{6}x^{6}+\frac{52}{25}x^{5}-\frac{39}{80}x^{4}-\frac{71}{10}x^{3}+\frac{79}{20}x^{2}+x-\frac{1}{10}\)
& [-1.5,11] & 13870 & 2974.180 &  \cite{Sergeyev:et:al.(2016a)} \\
\hline 2 & \(f(x) = - \sin x^{3} - \cos x^{3}\) & [0,6.28] & 2.2 &
-0.800 & \cite{molinaro2001finding} \\ \hline 3 & \(f(x) = x - \sin
3x  + 1 \) & [0,6.5] & 4 & 1.202 & \cite{molinaro2001finding} \\
\hline 4 & \(f(x) = \frac{x^{2}-5x+6}{x^{2}+1} \) & [-5,5] & 6.5 &
0.671 & \cite{Sergeyev:et:al.(2016a)} \\ \hline
\end{tabular}
}
\end{center}
}
%
{\footnotesize \caption{Test Case 2} \label{tab:TC2}
\begin{center}
{\renewcommand{\arraystretch}{2.4}
\vspace*{-1mm} \begin{tabular}{|c|c|c|c|c|c|} \hline \bf Problem &
\bf Lipschitz Function & \bf Interval & \bf \emph{L} & \bf Threshold
& \bf Ref.\\ \hline 5 & \(f(x) = - \sin x - \sin \frac{10}{3} x\) &
[2.7,7.5] & 4.29 & -0.609 & \cite{hansen1992global} \\ \hline 6 &
\(f(x) = (-3x+1.4) \sin18x \) & [0,1.2] & 36 & -1.271 &
\cite{hansen1992global} \\ \hline 7 & \(f(x) = (x + \sin x )
e^{-x^{2}}  \) & [-10,10] & 2.5 & -0.659 & \cite{hansen1992global}
\\ \hline 8 & \(f(x) =  - \sin x - \sin \frac{2}{3} x \) &
[3.1,20.4] & 1.7 & -1.483 & \cite{hansen1992global}
\\ \hline 9 & \(f(x) = e^{-x} \sin 2 \pi x \) & [0,4] & 6.5 & -0.347
& \cite{hansen1992global}\\ \hline 10 & \(f(x) = - e^{-x} \sin 2 \pi
x + 0.5 \) & [0,4] & 6.5 & -0.154 & \cite{molinaro2001finding}\\
\hline
\end{tabular}
}
\end{center}
}
\end{table}

\begin{table}[!ht]
{\footnotesize \caption{Test Case 3} \label{tab:TC3}
\begin{center}
{\renewcommand{\arraystretch}{2.3}
\vspace*{-2mm} \begin{tabular}{|c|c|c|c|c|c|} \hline \bf Problem &
\bf Lipschitz Function & \bf Interval & \bf \emph{L} & \bf Threshold
& \bf Ref. \\ \hline 11 & \(f(x) = \sum_{i=1}^{5} i \sin[(i+1)x+i]+3
\) & [-10,10] & 67 & -24.335 & \cite{molinaro2001finding} \\ \hline
12 & \(f(x) = \cos x - \sin 5x
+ 1  \) & [0,7] & 5.951 & -0.545 & \cite{molinaro2001finding}\\
\hline 13 & \(f(x) = \begin{cases}
                    \cos5x, \quad x \leq \frac{3}{2} \pi\\
                    \cos x, \quad \mathrm{otherwise}
                  \end{cases}\) & [0,18] & 4.999 & -0.800 & \cite{molinaro2001finding}\\ \hline
14 & \(f(x) = \begin{cases}
                    \sin x, \quad x \leq \pi\\
                    \sin5x, \quad \mathrm{otherwise}
                  \end{cases}\) & [-10,10] & 4.999 & -0.800 & \cite{molinaro2001finding}\\ \hline
15 & \(f(x) = - \sum_{i=1}^{5} \cos[(i+1)x] \) & [-10,10] & 18.119 &
-4.229 & \cite{molinaro2001finding} \\ \hline 16 & \(f(x) = x |\sin
x| + 6 \) & [-10,10] & 9.632 & -0.332 & \cite{molinaro2001finding}\\
\hline 17 & \(f(x) = |x \sin x| - 1.5 \) & [-10,10] & 9.632 & -0.709
& \cite{molinaro2001finding}\\ \hline 18 & \(f(x) = \begin{cases}
                    \sin x, \quad \sin x > \cos x\\
                    \cos x, \quad \mathrm{otherwise}
                  \end{cases}\) & [-10,10] & 1 & -0.519 & \cite{molinaro2001finding}\\ \hline
\end{tabular}
}
\end{center}
}
\end{table}

\begin{table}[t]
{\footnotesize \caption{The number of points at which the objective
function has been evaluated and the number of these evaluations
executed by Algorithms~1 and~2 on the three test cases}
\label{tab:Results}
\begin{center}
{\renewcommand{\arraystretch}{1.2}
\vspace*{-2mm} \begin{tabular}{|c|c|c|c|c|c|c|} \hline
\multicolumn{1}{|l|}{\multirow{2}{*}{\textbf{Problem}}} &
\multicolumn{2}{c|}{\textbf{Safe Expansion}} &
\multicolumn{2}{c|}{\textbf{Global Maximization}} &
\multicolumn{2}{c|}{\textbf{Total}}
\\
\cline{2-7} \multicolumn{1}{|l|}{} & \textbf{ Points} &
\textbf{Evaluations} & \textbf{ Points} & \textbf{Evaluations} &
\textbf{ Points} & \textbf{Evaluations} \\ \hline
%
%
1  & 15  & 41   & 20  & 31  & 35  & 72 \\
2  & 10  & 43   & 82  & 178 & 92  & 221 \\
3  & 16  & 44   & 10  & 19  & 26  & 63 \\
4  & 40  & 95   & 39  & 99  & 79  & 194 \\ \hline
5  & 52  & 191  & 29  & 33  & 81  & 224 \\
6  & 59  & 92   & 27  & 30  & 86  & 122 \\
7  & 150 & 209  & 128 & 209 & 278 & 418 \\
8  & 53  & 245  & 68  & 169 & 121 & 414 \\
9  & 380 & 1002 & 393 & 403 & 773 & 1495 \\
10 & 71  & 95   & 35  & 44  & 106 & 139 \\ \hline
11 & 210 & 485  & 170 & 196 & 380 & 681 \\
12 & 57  & 184  & 57  & 97  & 114 & 281 \\
13 & 99  & 441  & 55  & 74  & 154 & 515 \\
14 & 101 & 441  & 35  & 61  & 136 & 502 \\
15 & 93  & 255  & 89  & 108 & 182 & 363 \\
16 & 54  & 136  & 46  & 81  & 100 & 217 \\
17 & 59  & 249  & 74  & 169 & 133 & 418 \\
18 & 27  & 191  & 29  & 35  & 56  & 226 \\ \hline
\end{tabular}
}
\end{center}
}
\end{table}

\section{Conclusions} \label{sec:conclusions} The problem of safe
global maximization of an expensive   black-box function $f(x)$
satisfying the Lipschitz condition has been considered in this
paper. The word \emph{expensive} means here that each evaluation of
the objective function $f(x)$ is a time consuming operation.  With
respect to the traditional formulations used in global optimization
the problem under consideration here has two important distinctions:
  \begin{enumerate}
 \item[i.]
 The first difficulty consists in the    presence of the noise. As
 a result, instead of the function $f(x)$ a function $g(x)$
 distorted by the noise is optimized.
 \item[ii.]
The second difficulty is related to the notion of safe optimization.
The term ``safe" means that the objective function $g(x)$ during
optimization should not violate a ``safety" threshold, that in our
case is a certain a priori given value $h$ in the maximization
problem. Any new function evaluation must be performed at ``safe
points" only, namely, at points $y$ for which it is known that the
objective function $g(y) \geq h$. Clearly, the main difficulty here
consists in the fact that the used optimization algorithm should
ensure that the safety constraint will be satisfied at a point $y$
\textbf{\emph{before}} evaluation of $f(y)$ will be executed.
 \end{enumerate}

In our approach, the problem under investigation has been split in
two parts (phases):
 \begin{enumerate}
   \item
During the first phase, it is required  to find an approximation
$\widehat{\Omega}$ of the safe region $\Omega$ within the search
domain $D$ by learning from the information received from
evaluations of $g(x)$.
   \item
Then, during the second phase, it is  necessary to find an
approximation of the global maximum of $g(x)$ over
$\widehat{\Omega}$.
 \end{enumerate}

A theoretical study of the problem has been performed  and it has
been shown that even though the objective function $f(x)$ satisfies
the Lipschitz condition, traditional Lipschitz minorants and
majorants cannot be used in this context due to the presence of the
noise. In fact, a counterexample showing that the usage of simple
Lipschitz ideas can lead to erroneous evaluations of $g(x)$ at
unsafe points has been presented.

Then, a $\delta$-Lipschitz framework has been introduced and an
algorithm determining the safe area within the search domain $D$ has
been proposed. It has been theoretically proven that this method is
able to construct a minorant   ensuring  that all   executed
evaluations of $g(x)$ performed by the method will be safe. It has
been shown that the introduced procedure allows one to expand
reliably  the safe region from initial safe points and to obtain the
desired approximation $\widehat{\Omega}$ of the safe
region\,$\Omega$.

After that the second algorithm executing   the safe global
maximization over the found safe region $\widehat{\Omega}$ has been
proposed.  It has been theoretically investigated and conditions
allowing one to make conclusions with respect to not only the noisy
function $g(x)$ but also regarding the original function $f(x)$ have
been established. It should be stressed that the introduced
algorithm is able not only to provide lower and upper bounds for the
global maxima of $g(x)$ and $f(x)$ but to restrict significantly the
area of a possible location of the respective global maximizers, as
well.

Finally, numerical experiments executed on a series of problems
showing the reliability of the proposed procedures have been
performed.

\section*{Compliance with Ethical Standards}

\textbf{Conflict of interest:} The authors declare that they have no
conflict of interest.

\vspace*{2mm}

\noindent \textbf{Ethical approval:} This article does not contain
any studies with human participants or animals performed by any of
the authors.


\begin{figure}[!hp]
    \centering
    \subfloat[Test Function 1, Phase 1]{\label{fig:t1f11}\includegraphics[width=0.50\textwidth]{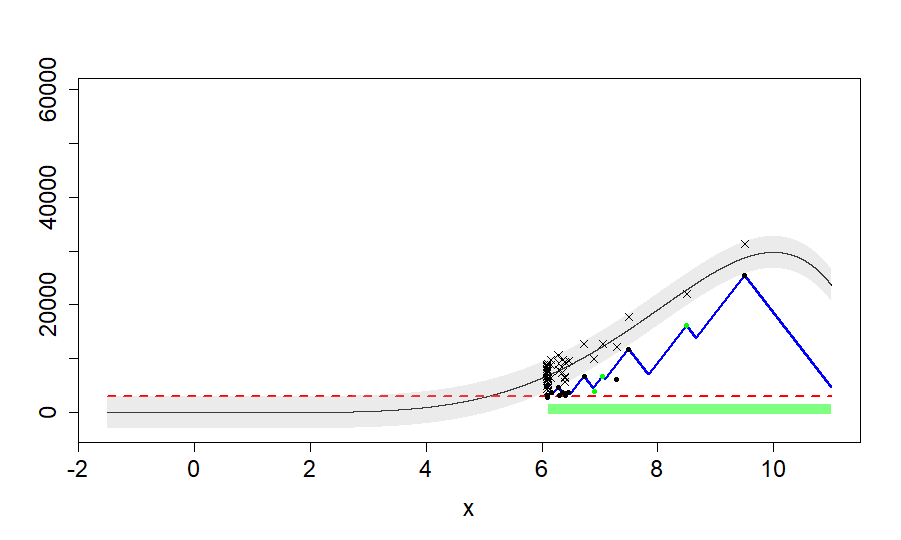}}
    \subfloat[Test Function 1, Phase 2]{\label{fig:t1f12}\includegraphics[width=0.50\textwidth]{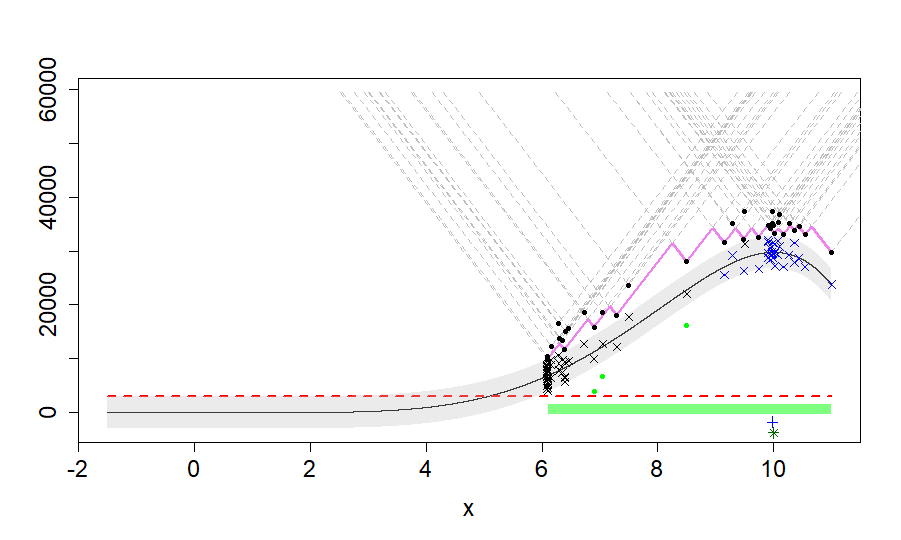}}\\
    \subfloat[Test Function 2, Phase 1]{\label{fig:t1f21}\includegraphics[width=0.50\textwidth]{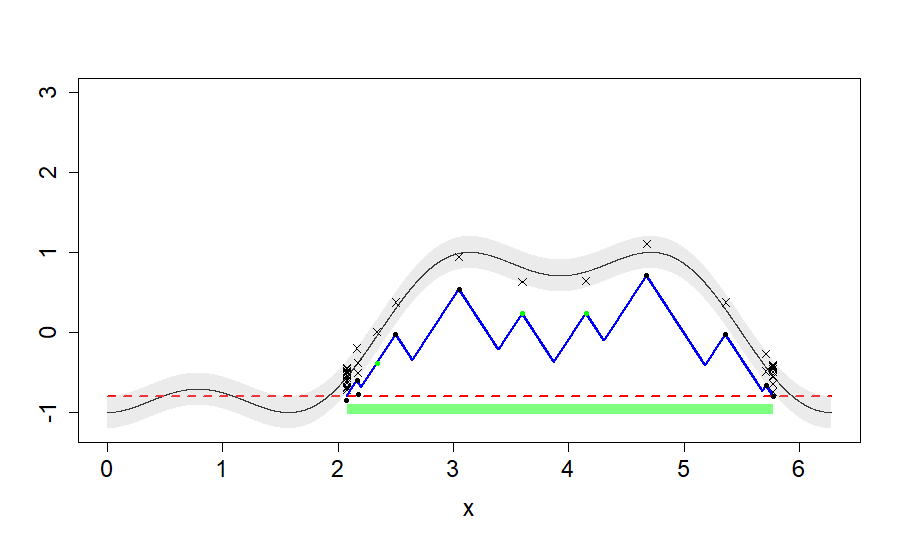}}
    \subfloat[Test Function 2, Phase 2]{\label{fig:t1f22}\includegraphics[width=0.50\textwidth]{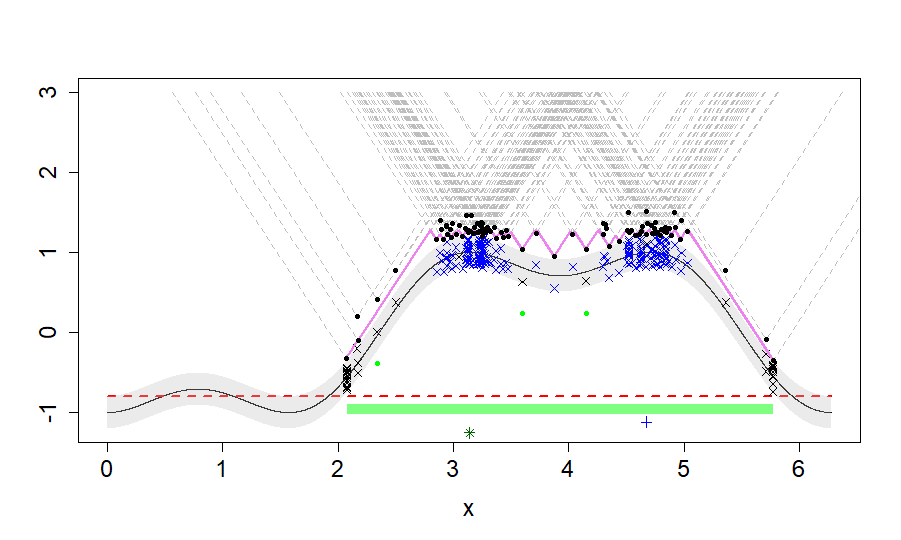}}\\
    \subfloat[Test Function 3, Phase 1]{\label{fig:t1f31}\includegraphics[width=0.50\textwidth]{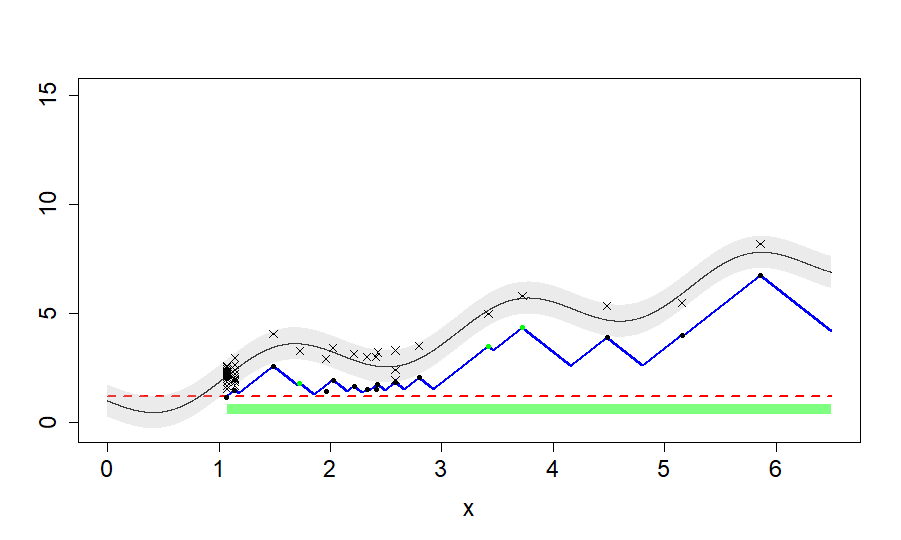}}
    \subfloat[Test Function 3, Phase 2]{\label{fig:t1f32}\includegraphics[width=0.50\textwidth]{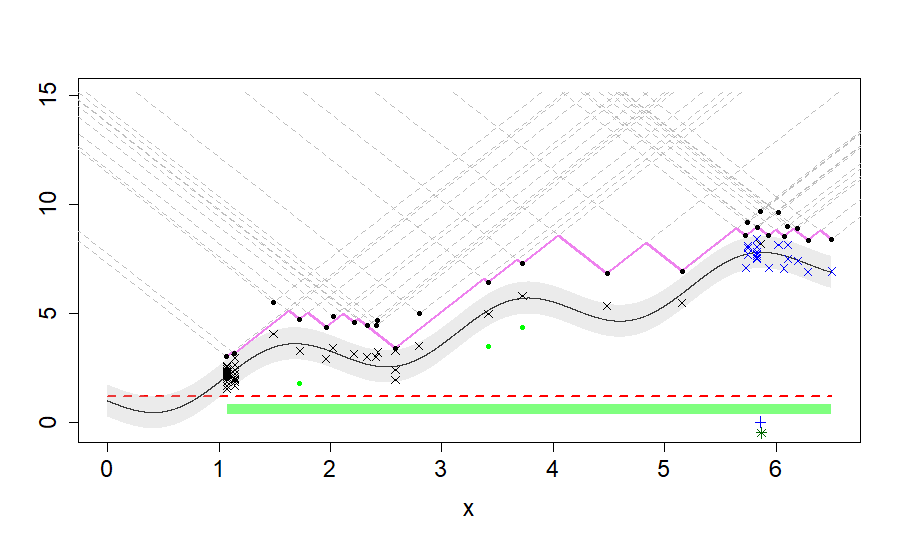}}\\
    \subfloat[Test Function 4, Phase 1]{\label{fig:t1f41}\includegraphics[width=0.50\textwidth]{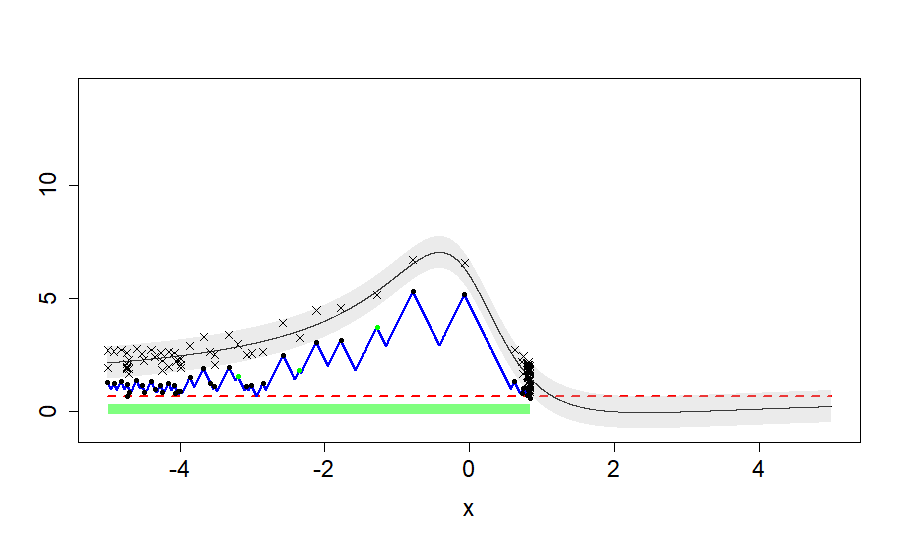}}
    \subfloat[Test Function 4, Phase 2]{\label{fig:t1f42}\includegraphics[width=0.50\textwidth]{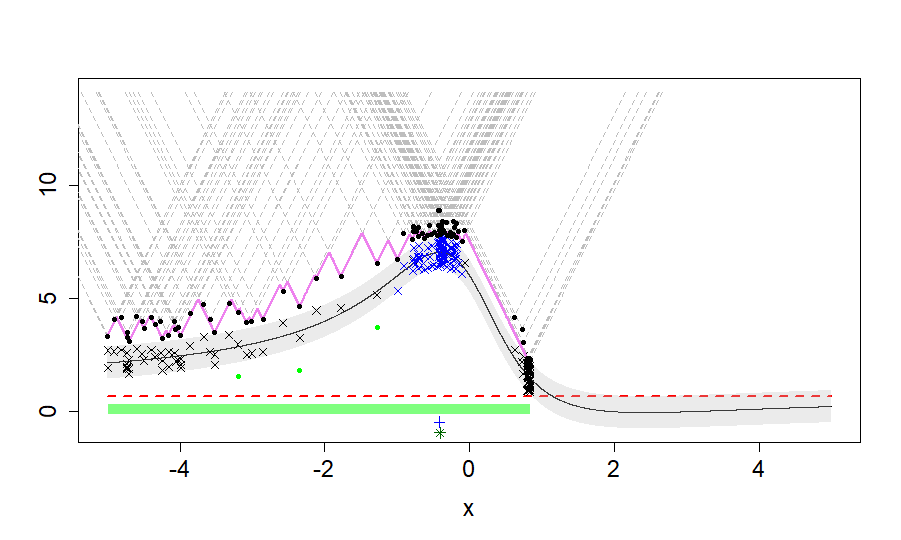}}\\
\caption{Graphical results of the $\delta$-Lipschitz framework on
Test Case 1} \label{fig:TestCase1Part1}
\end{figure}

\newpage

\begin{figure}[!hp]
    \centering
    \subfloat[Test Function 5, Phase 1]{\label{fig:t2f51}\includegraphics[width=0.50\textwidth]{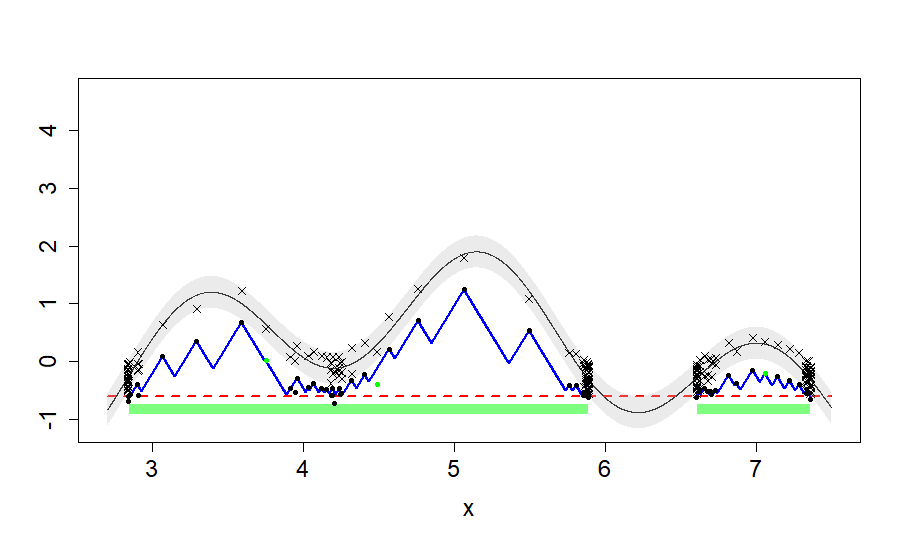}}
    \subfloat[Test Function 5, Phase 2]{\label{fig:t2f52}\includegraphics[width=0.50\textwidth]{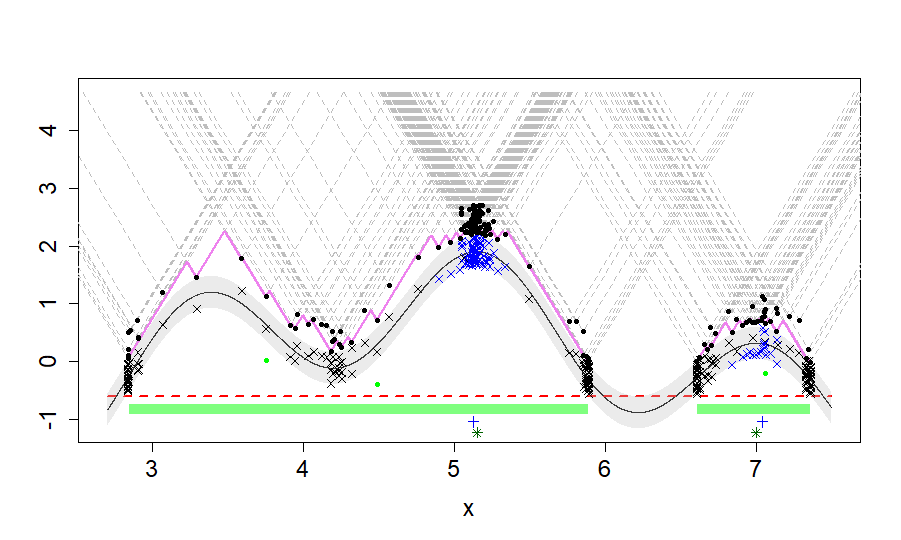}}\\
    \subfloat[Test Function 6, Phase 1]{\label{fig:t2f61}\includegraphics[width=0.50\textwidth]{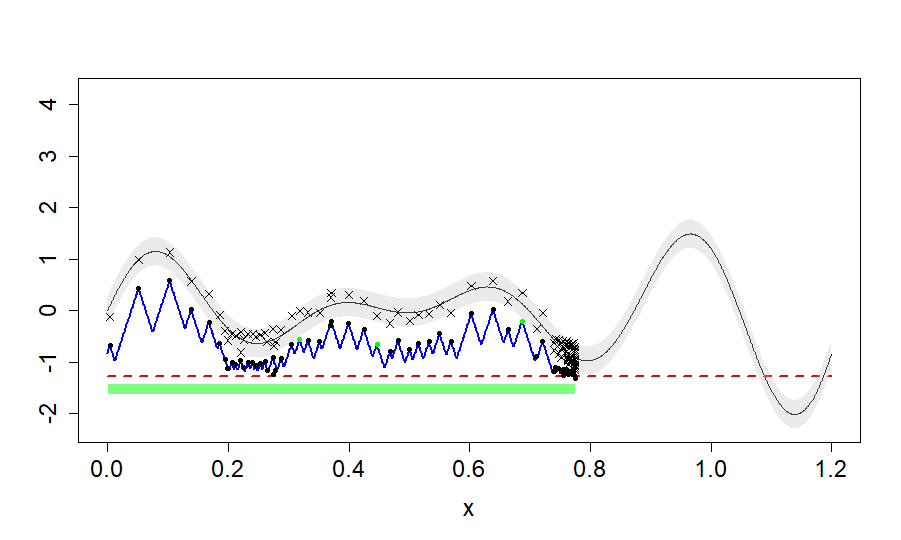}}
    \subfloat[Test Function 6, Phase 2]{\label{fig:t2f62}\includegraphics[width=0.50\textwidth]{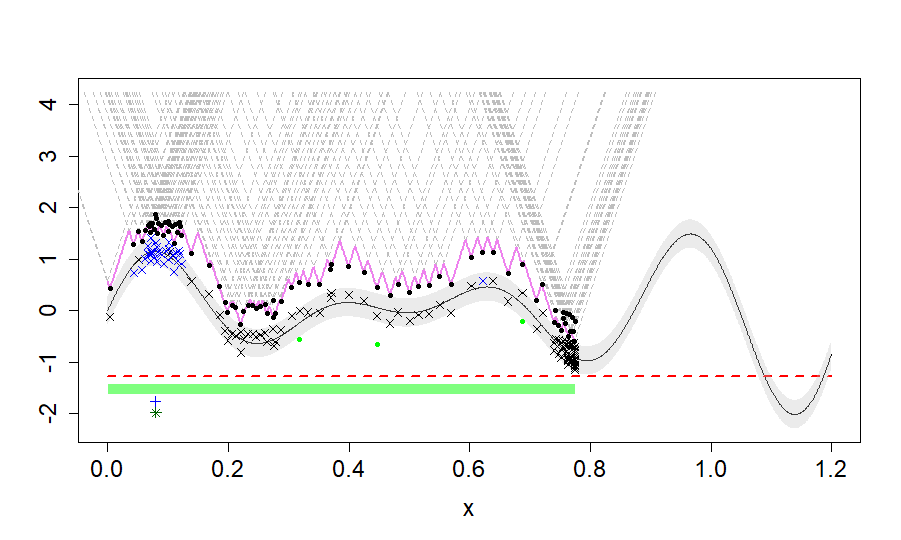}}\\
    \subfloat[Test Function 7, Phase 1]{\label{fig:t2f71}\includegraphics[width=0.50\textwidth]{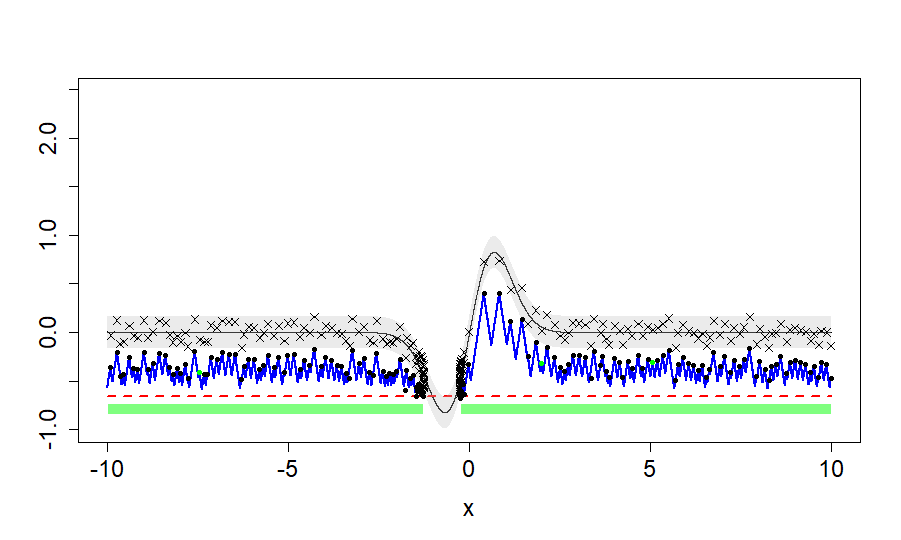}}
    \subfloat[Test Function 7, Phase 2]{\label{fig:t2f72}\includegraphics[width=0.50\textwidth]{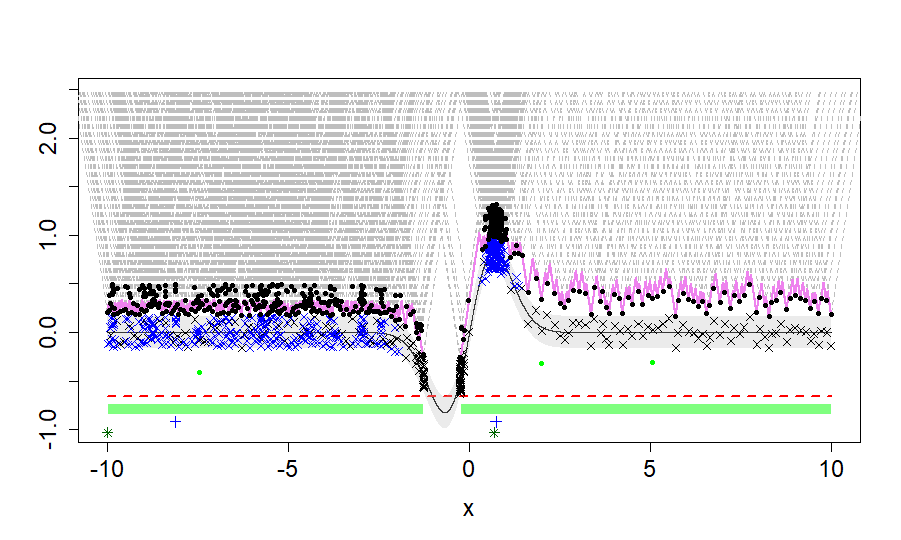}}\\
    \subfloat[Test Function 8, Phase 1]{\label{fig:t2f81}\includegraphics[width=0.50\textwidth]{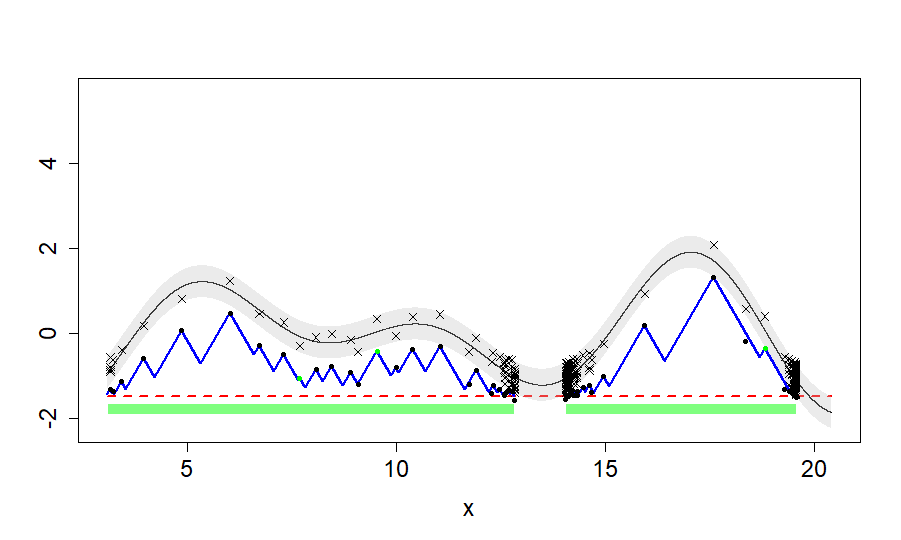}}
    \subfloat[Test Function 8, Phase 2]{\label{fig:t2f82}\includegraphics[width=0.50\textwidth]{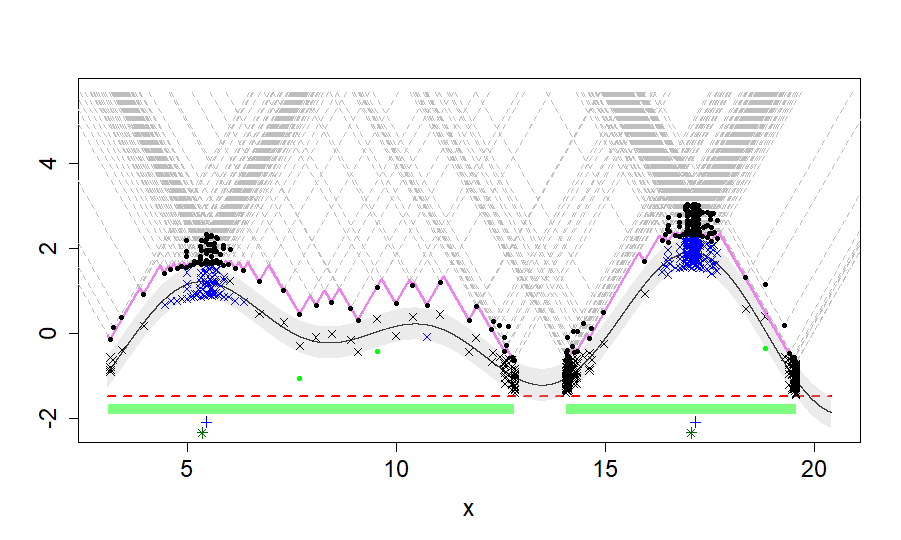}}\\
\caption{Graphical results of the $\delta$-Lipschitz framework on
Test Case 2, Part 1} \label{fig:TestCase2Part1}
\end{figure}

\newpage

\begin{figure}[!hp]
    \centering

    \subfloat[Test Function 9, Phase 1]{\label{fig:t2f91}\includegraphics[width=0.50\textwidth]{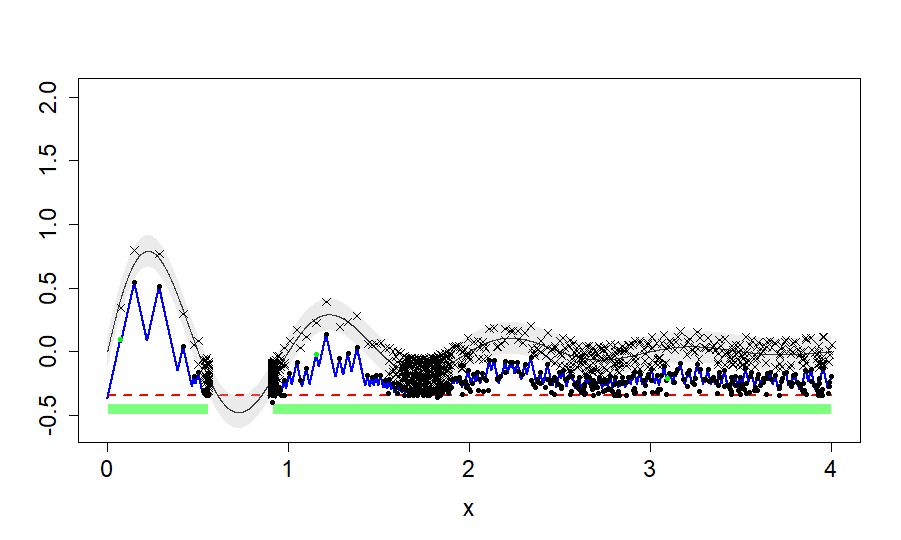}}
    \subfloat[Test Function 9, Phase 2]{\label{fig:t2f92}\includegraphics[width=0.50\textwidth]{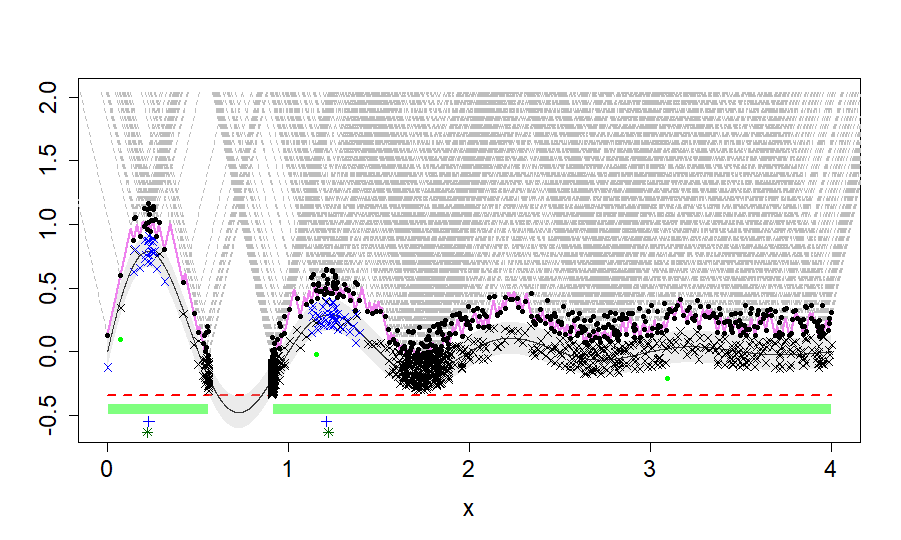}}\\
    \subfloat[Test Function 10, Phase 1]{\label{fig:t2f101}\includegraphics[width=0.50\textwidth]{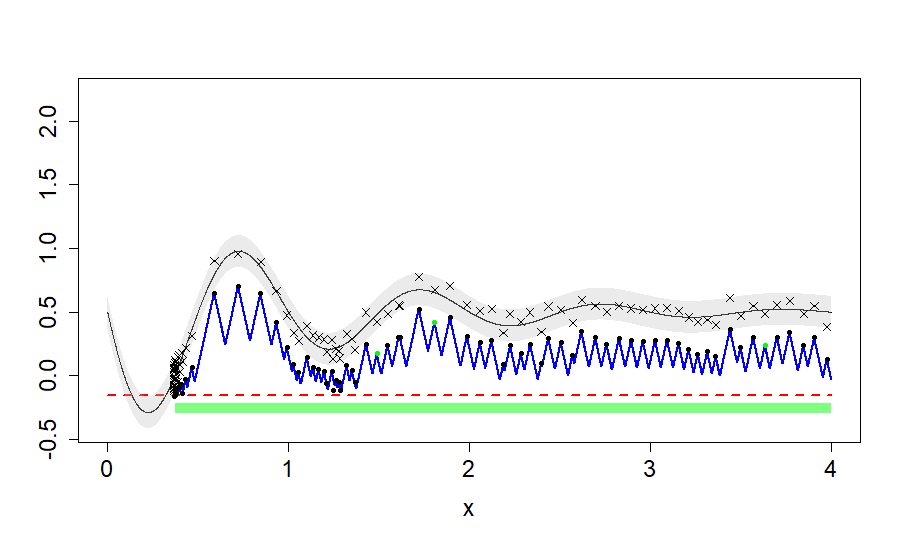}}
    \subfloat[Test Function 10, Phase 2]{\label{fig:t2f102}\includegraphics[width=0.50\textwidth]{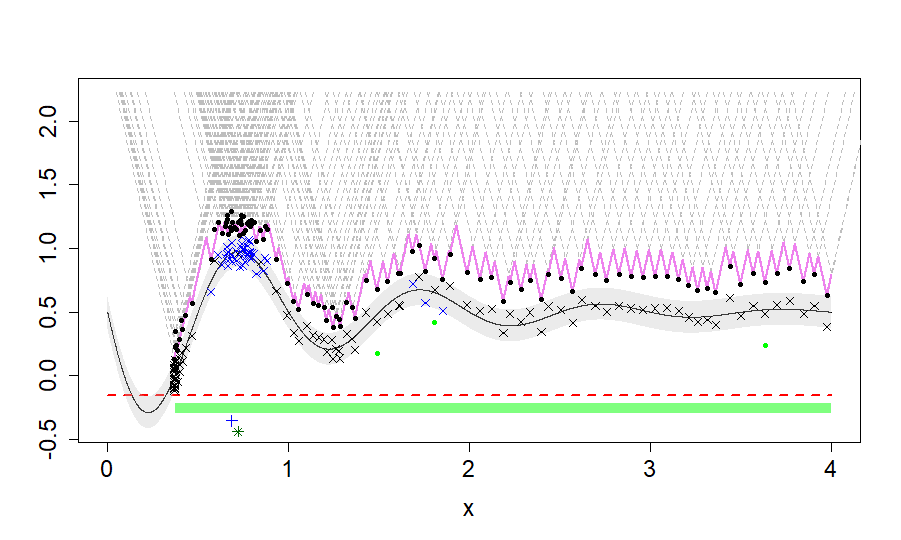}}\\
\caption{Graphical results of the $\delta$-Lipschitz framework on
Test Case 2, Part 2} \label{fig:TestCase2Part2}
\end{figure}

\newpage

\begin{figure}[!hp]
    \centering
    \subfloat[Test Function 11, Phase 1]{\label{fig:t3f111}\includegraphics[width=0.50\textwidth]{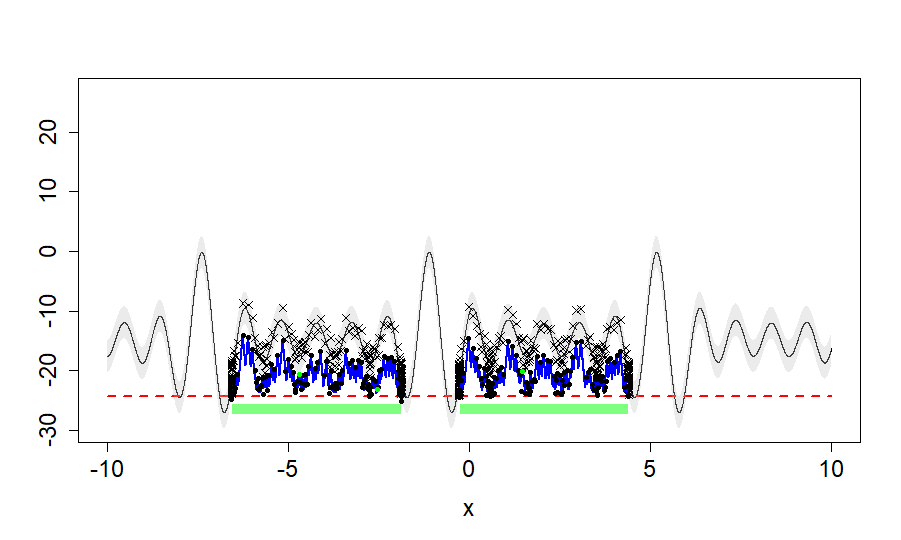}}
    \subfloat[Test Function 11, Phase 2]{\label{fig:t3f112}\includegraphics[width=0.50\textwidth]{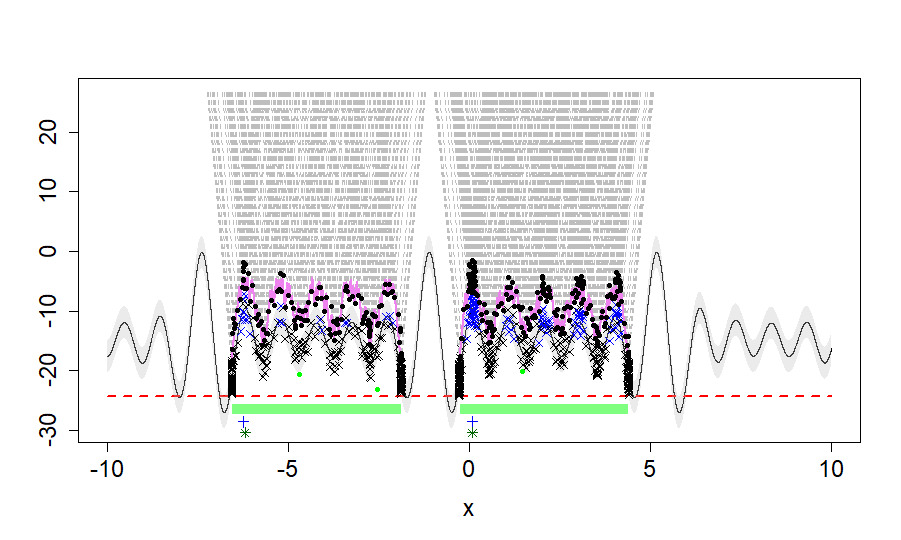}}\\
    \subfloat[Test Function 12, Phase 1]{\label{fig:t3f121}\includegraphics[width=0.50\textwidth]{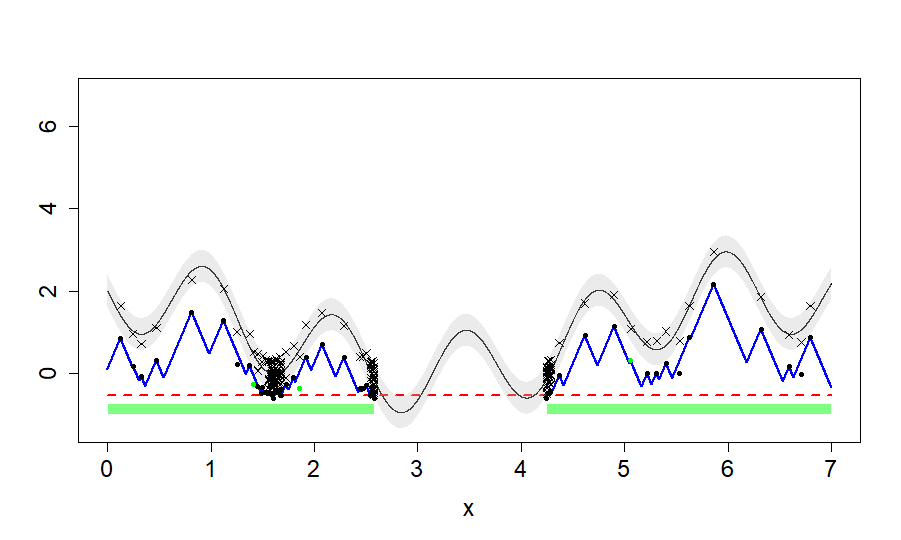}}
    \subfloat[Test Function 12, Phase 2]{\label{fig:t3f122}\includegraphics[width=0.50\textwidth]{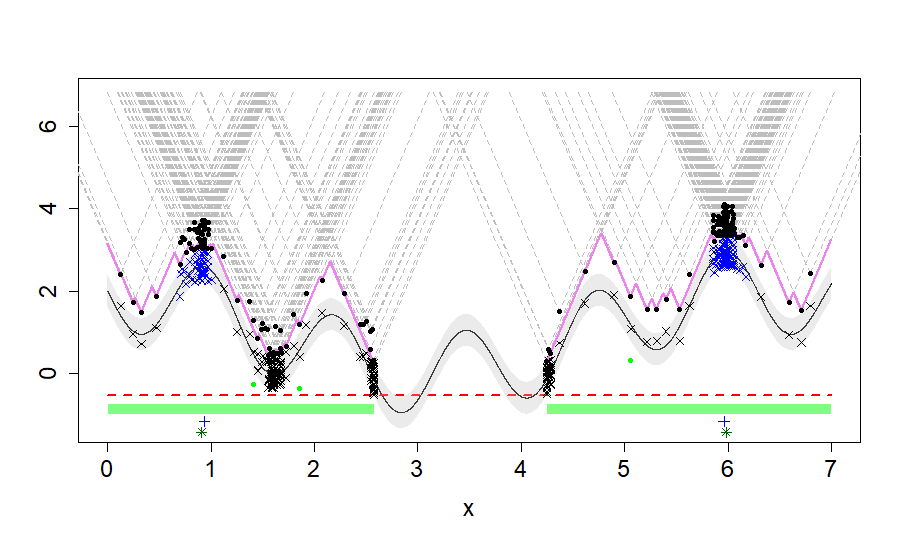}}\\
    \subfloat[Test Function 13, Phase 1]{\label{fig:t3f131}\includegraphics[width=0.50\textwidth]{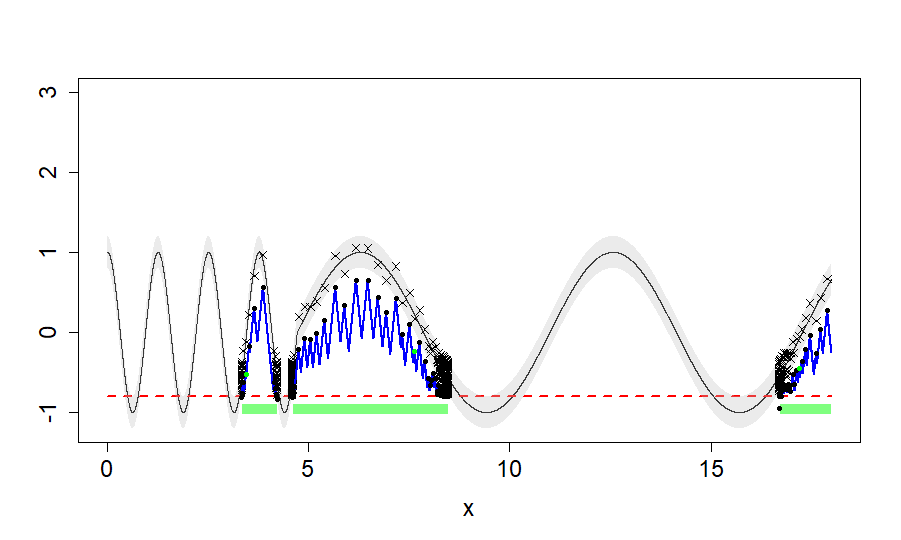}}
    \subfloat[Test Function 13, Phase 2]{\label{fig:t3f132}\includegraphics[width=0.50\textwidth]{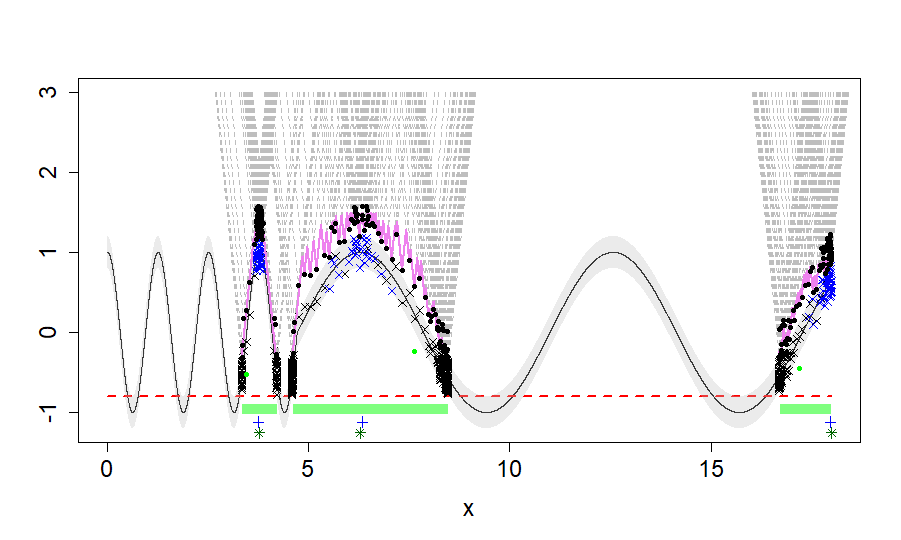}}\\
    \subfloat[Test Function 14, Phase 1]{\label{fig:t3f141}\includegraphics[width=0.50\textwidth]{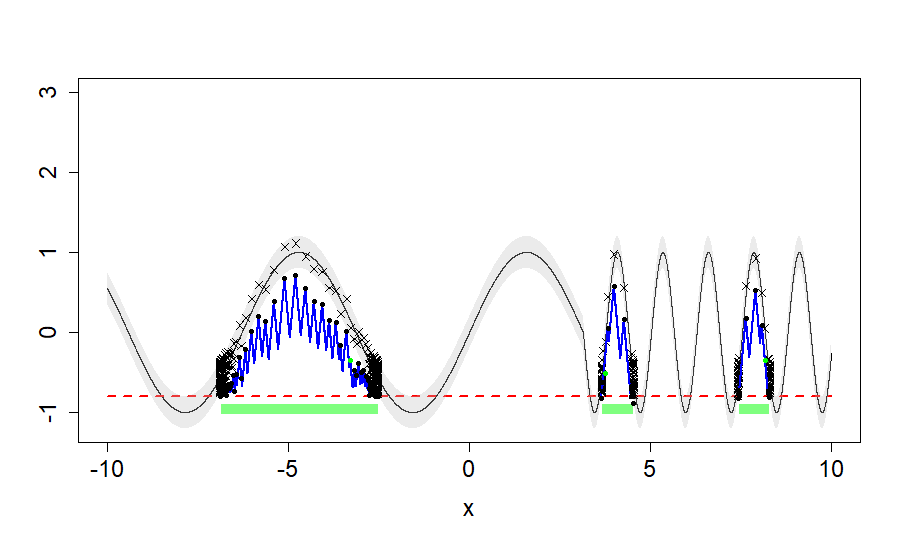}}
    \subfloat[Test Function 14, Phase 2]{\label{fig:t3f142}\includegraphics[width=0.50\textwidth]{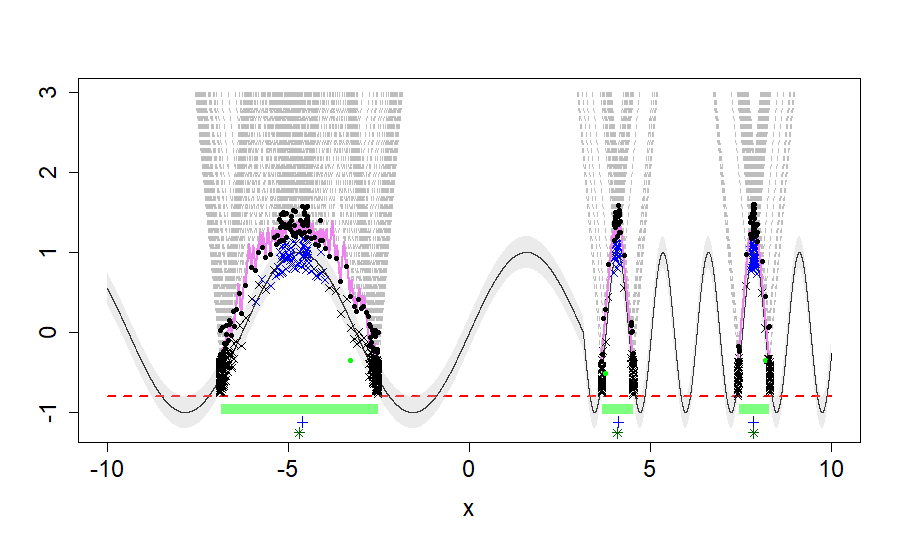}}\\
\caption{Graphical results of the $\delta$-Lipschitz framework on
Test Case 3, Part 1} \label{fig:TestCase3Part1}
\end{figure}

\newpage

\begin{figure}[!hp]
    \centering
    \subfloat[Test Function 15, Phase 1]{\label{fig:t3f151}\includegraphics[width=0.50\textwidth]{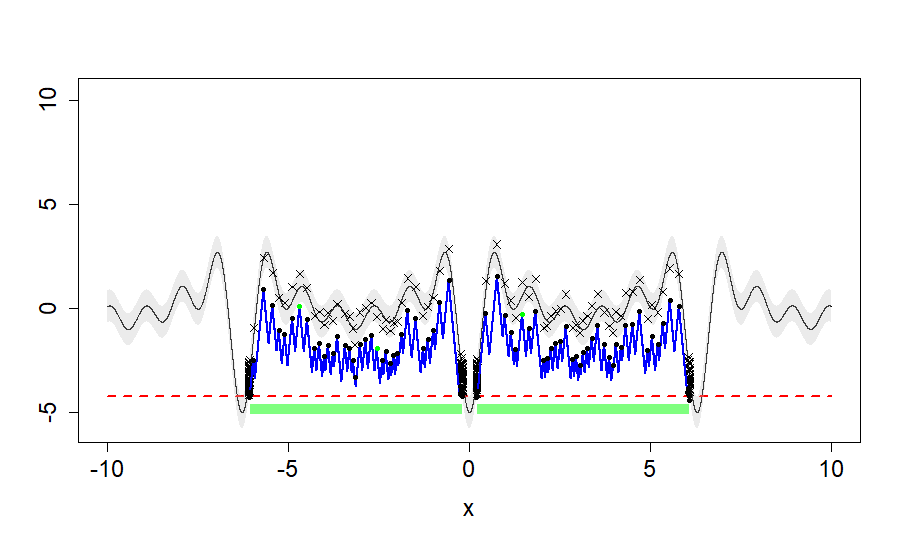}}
    \subfloat[Test Function 15, Phase 2]{\label{fig:t3f152}\includegraphics[width=0.50\textwidth]{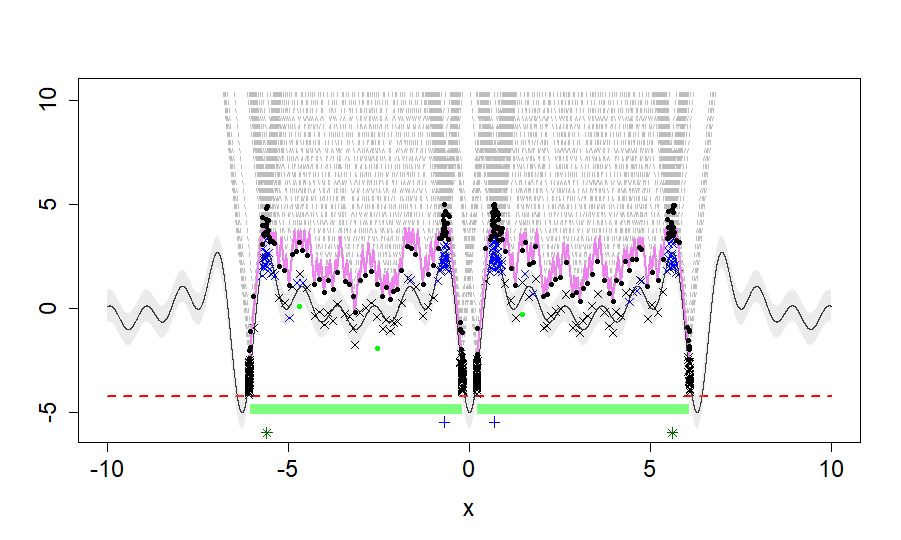}}\\
    \subfloat[Test Function 16, Phase 1]{\label{fig:t3f161}\includegraphics[width=0.50\textwidth]{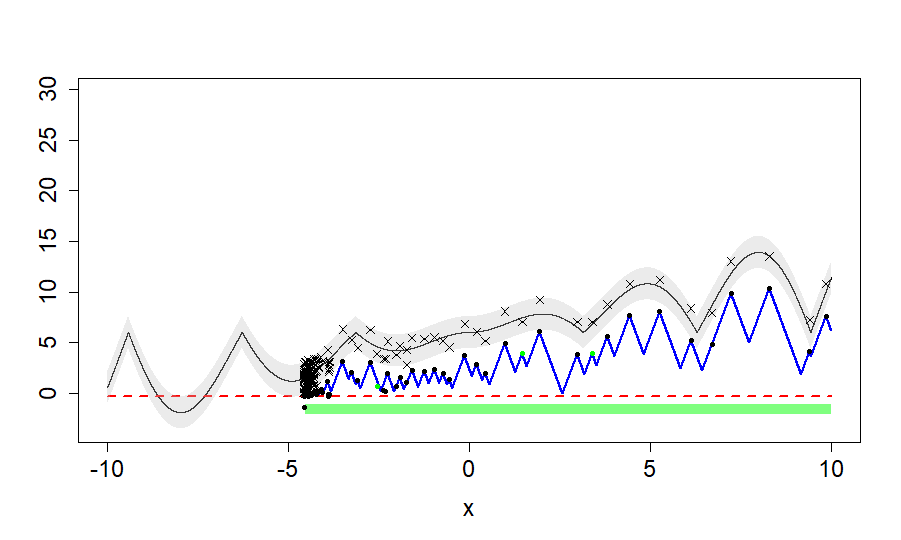}}
    \subfloat[Test Function 16, Phase 2]{\label{fig:t3f162}\includegraphics[width=0.50\textwidth]{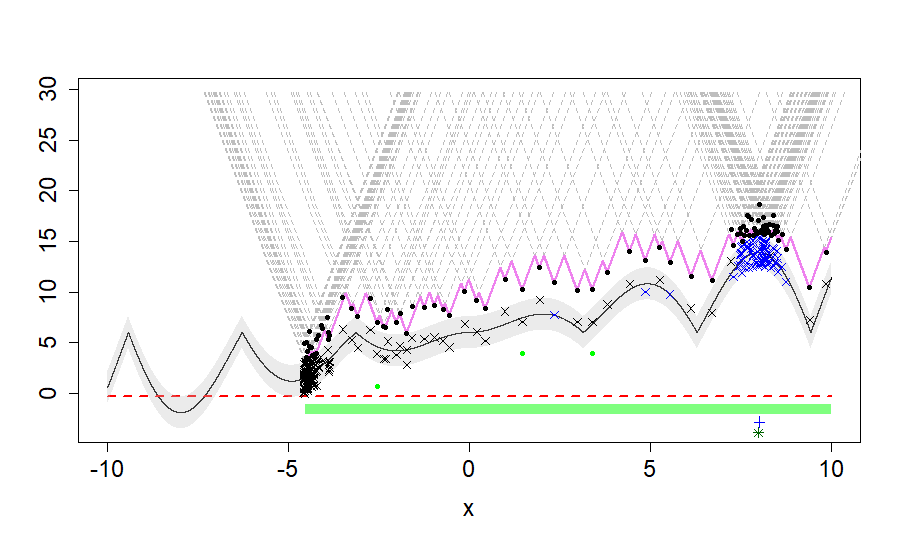}}\\
    \subfloat[Test Function 17, Phase 1]{\label{fig:t3f171}\includegraphics[width=0.50\textwidth]{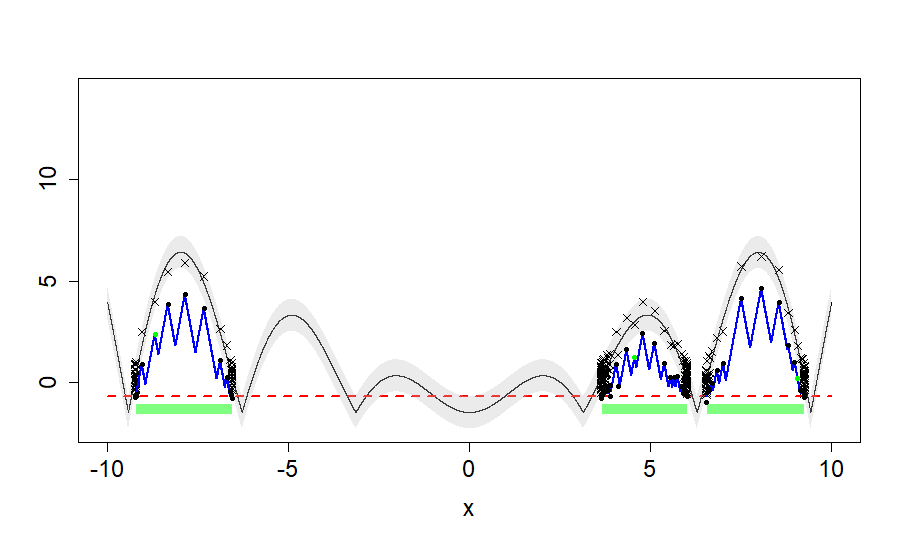}}
    \subfloat[Test Function 17, Phase 2]{\label{fig:t3f172}\includegraphics[width=0.50\textwidth]{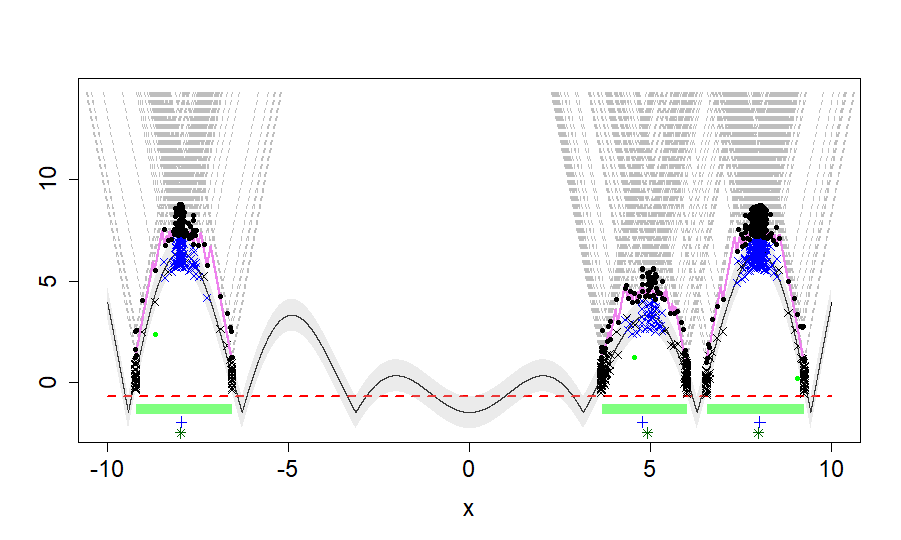}}\\
    \subfloat[Test Function 18, Phase 1]{\label{fig:t3f181}\includegraphics[width=0.50\textwidth]{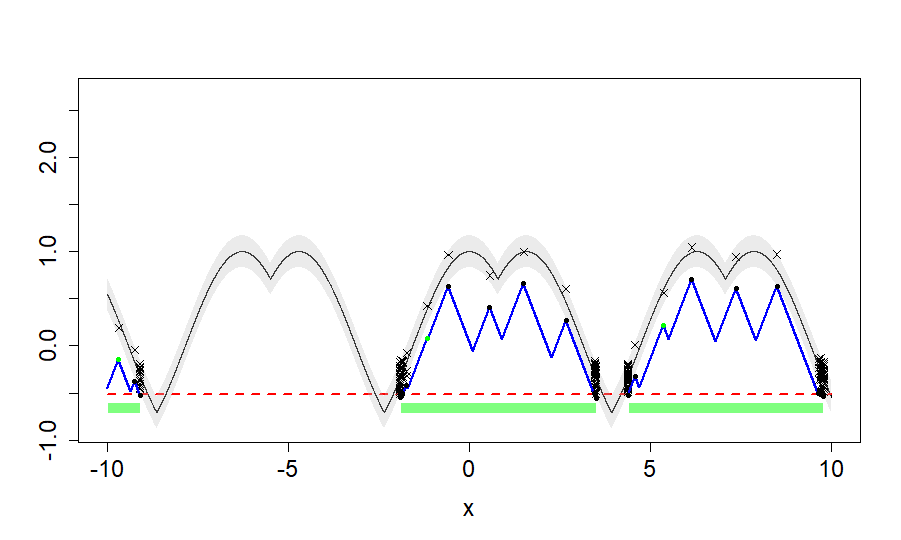}}
    \subfloat[Test Function 18, Phase 2]{\label{fig:t3f182}\includegraphics[width=0.50\textwidth]{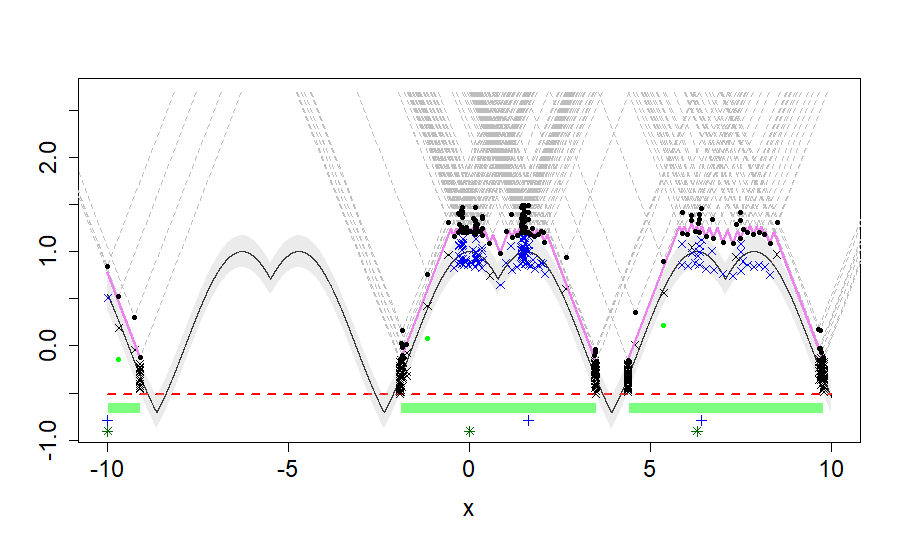}}\\
\caption{Graphical results of the $\delta$-Lipschitz framework on
Test Case 3, Part 2} \label{fig:TestCase3Part2}
\end{figure}

\newpage

\bibliographystyle{spmpsci}
\bibliography{Xbib_Safe}

\appendix
\begin{algorithm}[t]
\caption{A reliable expansion of the safe region}
\label{alg:safe_expansion}
\begin{algorithmic}[1]
\State{INPUT:}
\State{$D=[a,b]$ search space}
\State{$L$ Lipschitz constant}
\State{$\delta$ maximal quantity of noise: $|\xi(x)|\leq\delta$}
\State{$h$ safety threshold}
\State{$S_{i=1:n}^0$ the set of $n$ initial safe points, with $S_i^k=[\alpha_i^k,\beta_i^k]$, where $\alpha_i=\beta_i=\bar{x_i}$. Safe expansion starts from every initial safe point leading to the identification of an equal or lower number of safe regions}
\State{$r_{\alpha_i}, r_{\beta_i}$ the sets of noisy function values observed at $\alpha_i$ and $\beta_i$, respectively, namely \emph{repetitions}}
\State{$\varepsilon$ minimal expansion (if the expansion induced by the new function evaluation is lower than $\varepsilon$, then the function evaluation is considered as a \emph{repetition} at the current point)}
\State{$\nu$ maximum number of repetitions (stop criterion for expansion - depends on $\varepsilon$})
\State{$\sigma$ tolerance (stop criterion for expansion)}

\State{OUTPUT:}
\State{$\Hat{\Omega} = {\cup}_j \hat{\Omega}_j$
approximation of the safe region $\Omega$}
\State{$P$ the set of
trial points (for the Global Maximization (Algorithm 2))}

\State{INITIALIZATION:}
\State{$A^k,B^k=\emptyset$, where: $A^k=\{i=1,..,n: \alpha_i^k\}$ cannot be expanded any longer and $B^k=\{i=1,..,n: \beta_i^k\}$ cannot be expanded any longer}
\State{$k\leftarrow 0$ (current iteration)}
\State{$r_{\alpha_i} \leftarrow r_{\beta_i} \leftarrow \emptyset$}
\State{$P \leftarrow S_{1:n}^0$}
\State{}

\For{$i=1:n$}
    \State{$\#$ \emph{merging overlapping safe regions} }
    \If{$\exists S_j^k: \alpha_i^k<\beta_j^k, j=1,...,n, j\neq i$}
        \State{$A^k \leftarrow A^k \cup \{i\}$}
    \EndIf
    \If{$\exists S_j^k: \beta_i^k>\alpha_j^k, j=1,...,n, j\neq i$}
        \State{$B^k \leftarrow B^k \cup \{i\}$}
    \EndIf

    \State{$\#$ \emph{safe expansion towards extreme a}}
    \If{$g(\alpha_i^k)-2\delta >h$}
        \State{$\alpha_i^{k+1} \leftarrow \alpha_i^k - \frac{|g(\alpha_i^k)-2\delta-h|}{L}$}
        \If{$|\alpha_i^{k+1}-\alpha_i^k|<\varepsilon$}
            \State{$r_{\alpha_i}\leftarrow r_{\alpha_i} \cup \{g(\alpha_i^k)\}$}
        \Else
            \State{$r_{\alpha_i} \leftarrow \emptyset$}
            \State{$P \leftarrow P \cup max\{a,\alpha_i^{k+1}\}$}
        \EndIf
    \algstore{savename}
\end{algorithmic}
\end{algorithm}

\begin{algorithm}
\begin{algorithmic}[1]
    \algrestore{savename}
    \Else
        \State{$\alpha_i^{k+1} \leftarrow \alpha_i^k$}
        \State{$r_{\alpha_i} \leftarrow r_{\alpha_i} \cup \{g(\alpha_i^k)\}$}
    \EndIf
    \If{$|r_{\alpha_i}|==\nu \lor max_i\{r_{\alpha_i}\}-min_i\{r_{\alpha_i}\}>2\delta - \sigma \lor \alpha_i^k<a$}
        \State{$A^{k+1} \leftarrow A^k \cup \{i\} $}
    \EndIf

    \State{$\#$ \emph{safe expansion towards extreme b}}
    \If{$g(\beta_i^k)-2\delta >h$}
        \State{$\beta_i^{k+1} \leftarrow \beta_i^k + \frac{|g(\beta_i^k)-2\delta-h|}{L}$}
        \If{$|\beta_i^{k+1}-\beta_i^k|<\varepsilon$}
            \State{$r_{\beta_i}\leftarrow r_{\beta_i} \cup \{g(\beta_i^k)\}$}
        \Else
            \State{$r_{\beta_i} \leftarrow \emptyset$}
            \State{$P \leftarrow P \cup min\{b,\beta_i^{k+1}\} $}
        \EndIf
    \Else
        \State{$\beta_i^{k+1} \leftarrow \beta_i^k$}
        \State{$r_{\beta_i} \leftarrow r_{\beta_i} \cup \{g(\beta_i^k)\}$}
    \EndIf
    \If{$|r_{\beta_i}|==\nu \lor max_i\{r_{\beta_i}\}-min_i\{r_{\beta_i}\}>2\delta - \sigma \lor \beta_i^k<a$}
        \State{$B^{k+1} \leftarrow B^k \cup \{i\} $}
    \EndIf
\EndFor
\If{$|A^k|\neq n \lor |B^k|\neq n$}
    \State{$k \leftarrow k+1$}
    \State{(GO BACK TO THE For LOOP (line 20))}
\Else
    \State{$\#$ \emph{finally, merge possible overlapping safe regions}}
    \State{$S_i \leftarrow S_i^k$}
    \State{$\bar{A}=\{\alpha_1\} \cup \{\alpha_l : \alpha_l > \beta_{l-1}, l=2,...,n\}$}
    \State{$\bar{B}=\{\beta_n\} \cup \{\beta_l : \beta_l < \alpha_{l+1}, l=n-1,...,1\}$}
    \State{$\hat{\Omega}_j=[\alpha_j \in \bar{A}, \beta_j \in \bar{B}]$}
\EndIf
\Statex{}
\Return
$\hat{\Omega}={\cup}_j \hat{\Omega}_j$ where $\hat{\Omega}$ is the found approximation of the safe region $\Omega$, which can consist of several disjoint safe subregions.
$P$, the set of trial points
\end{algorithmic}
\end{algorithm}

\begin{algorithm} [t]
\caption{Global Maximization}
\label{alg:global_maximization}
\begin{algorithmic}[1]
\State{INPUT:}
\State{$L$ Lipschitz constant}
\State{$\delta$ the maximal quantity of noise (as in Algorithm 1), see (3)}
\State{$\varepsilon$ accuracy of the global maximization (as in Algorithm 1)}
\State{$\nu$ the maximal number of allowed repetitions of evaluations of $g(x)$ (stopping criteria for expansion)}
\State{$P$, the set of safe trial points from Algorithm 1}
\State{$\Hat{\Omega} = {\cup}_j \hat{\Omega}_j$ the safe region found by Algorithm 1}
\State{$\Hat{\Omega}_j =[\alpha_j,\beta_j]$ the $j$-th disconnected safe subregion, with $j=1,\ldots,N_{\hat{\Omega}}$}
\State{FUNCTIONS:}
\State{$R_i$ characteristics of the $i$-th interval $[x_{i-1},x_i]$, with $i=2:k$, as defined in (27)}
\State{$\bar{x}_t$ as defined in (28),\,(30)}
\State{$\gamma(x,x_i)$, $i=1:k$, as defined in (21)}
\State{$\hat{g}(x_i)$ and $\check{g}(x_i),i=1:k$, as defined in (18)
and (19), respectively.}
    \State{} \For{$j = 1:N_{\hat{\Omega}}$}
    \State{$x_1,\ldots,x_k \in P : \forall q=1,\ldots,k \implies x_q \geq \alpha_j \land x_q \leq \beta_j$}
    \State{}
    \State{$\#$ \emph{Construction of $\Gamma_k(x)$ on points from Algorithm 1}}
    \For{$l=1:k$}
        \State{$z_l\leftarrow \gamma(x_l,x_1)$}
    \EndFor
    \For{$i=2:k$}
        \If{$\gamma(x_i,x_i)<z_i$}
            \For{$l=1:k$}
                \If{$\gamma(x_l,x_i)<z_l$}
                    \State{$z_l\leftarrow \gamma(x_l,x_i)$}
                \EndIf
            \EndFor
        \EndIf
    \EndFor
    \State{}

    \State{$\#$ \emph{Construction of $\Gamma_{k+1}(x)$ from $\Gamma_k(x)$} }
    \State{$R_{max}\leftarrow R_2$}
    \State{$t\leftarrow2$}
    \State{$\nu_{cur}\leftarrow 0$}
    \For{$i=3:k$}
        \State{$R_i \leftarrow 0.5(z_{i-1}+z_i)+0.5L(x_i-x_{i-1})$}
        \If{$(R_{max}<R_i)$}
            \State{$R_{max} \leftarrow R_i$}
            \State{$t \leftarrow i$}
        \EndIf
    \EndFor
            \algstore{safe_global_optimization}
\end{algorithmic}
\end{algorithm}

\begin{algorithm}
\begin{algorithmic}[1]
        \algrestore{safe_global_optimization}
    \If{$(x_t-x_{t-1}>\varepsilon)$}
        \State{$x^{k+1}\leftarrow\Bar{x}_t$}
        \State{$\nu_{cur}\leftarrow 1$}
        \State{$\Check{g}(x^{k+1})\leftarrow g(x^{k+1})$}
        \If{$\,\,(\Check{g}(x^{k+1})>R_{max})$}
            \If{$(\nu_{cur}+1=\nu)$}
                \State{\textbf{STOP} (The maximum allowed number $\nu$ of repetitions has been reached)}
            \Else
                \State{$\nu_{cur}\leftarrow \nu_{cur}+1$}
            \EndIf
            \State{\textbf{GOTO 46} (Re-evaluate $g(x^{k+1})$)}
        \Else
            \For{$i=k:t-1$}
                \State{$x_{i+1}\leftarrow x_i$}
                \State{$z_{i+1}\leftarrow z_i$}
            \EndFor
            \State{$x_t\leftarrow x^{k+1}$}
            \State{$z_t\leftarrow \Check{g}(x^{k+1})$}
            \For{$i=1:k+1$}
                \If{$(z_i>\gamma(x_i,x_t))$}
                    \State{$z_i \leftarrow \gamma(x_i,x_t)$}
                \EndIf
            \EndFor
            \State{$k\leftarrow k+1$}
            \State{\textbf{GOTO 18} (Continue with the next iteration within the same safe region $\Hat{\Omega}_j$)}
        \EndIf
    \Else
        \State{\textbf{STOP} (The required accuracy $\varepsilon$ has been reached)}
    \EndIf
\EndFor{ (Continue the search within the next safe region)}
 \Statex{}
 \Return
 the values $x^{*}_k$ and $g^{*}_k$; the majorant $\Gamma_k(x)$, the regions  $N^k_g$ and $N^k_f$
\end{algorithmic}
\end{algorithm}

\end{document}